\documentclass[a4paper,12pt,reqno]{article}
\usepackage[english]{babel}

\usepackage{amsmath,amsfonts,enumerate,amssymb,amsthm,color}
\usepackage{pgf,tikz}

\usetikzlibrary{arrows.meta}
\usetikzlibrary{calc}
\tikzset{edge/.style={-{Latex[scale=1.9]}}}

\usepackage{subcaption}

\newcommand{\sqr}{\mathbin{\square}}
\DeclareMathOperator{\Circ}{Circ}
\DeclareMathOperator{\Cay}{Cay}
\DeclareMathOperator{\Dih}{Dih}

\usepackage[hyphens]{url}
\usepackage[colorlinks=true,citecolor=black,linkcolor=black,urlcolor=blue]{hyperref}
\usepackage[hyphenbreaks]{breakurl}

\usepackage{listings}
\definecolor{mauve}{rgb}{0.58,0,0.82}
\definecolor{dkgreen}{rgb}{0,0.6,0}
\lstdefinestyle{pitonche} {
    language = Python,
    basicstyle = footnotesizettfamily,
    showspaces = false,
    showstringspaces = false,
    breakautoindent = true,
    flexiblecolumns = true,
    keepspaces = true,
    stepnumber = 1,
    xleftmargin = 0pt
}
\usepackage{url}
\theoremstyle{plain}
\newtheorem{theorem}{Theorem}[section]
\newtheorem{corollary}[theorem]{Corollary}
\newtheorem{proposition}[theorem]{Proposition}

\newtheorem{problem}[theorem]{Problem}
\newtheorem{lemma}[theorem]{Lemma}

\theoremstyle{remark}
\newtheorem*{remark}{Remark}

\newtheorem{example}[theorem]{Example}

\def\Z{\mathbb{Z}}

\newcommand{\cB}{\mathcal{B}}

\newcommand{\arxiv}[2]{\href{https://arxiv.org/abs/#1}{\texttt{arXiv:#1}} \texttt{[#2]}}

\DeclareMathOperator{\bicirc}{BiCirc}

\usepackage{authblk}

\begin{document}

\title{Classification of quartic bicirculant nut graphs}

\author[1,2]{Ivan Damnjanović}
\author[1,3,5]{Nino Bašić}
\author[1,3,5]{Tomaž Pisanski}
\author[4,5]{Arjana Žitnik}
\affil[1]{FAMNIT, University of Primorska, Koper, Slovenia}
\affil[2]{Faculty of Electronic Engineering, University of Niš, Niš, Serbia}
\affil[3]{IAM, University of Primorska, Koper, Slovenia}
\affil[4]{FMF, University of Ljubljana, Ljubljana, Slovenia}
\affil[5]{Institute of Mathematics, Physics and Mechanics, Ljubljana, Slovenia}

\maketitle

\begin{abstract}
A graph is called a \emph{nut graph} if zero is its eigenvalue of multiplicity one and its corresponding eigenvector has no zero entries.
A graph is a \emph{bicirculant} if it admits an automorphism with two equally sized vertex orbits.
There are four classes of connected quartic bicirculant graphs.
We classify the quartic bicirculant graphs that are nut graphs
by investigating properties of each of these four classes.


\end{abstract}

\section{Introduction}
\label{sec:intro}

The motivation for this paper comes from the recent interest for \emph{nut graphs}, i.e., graphs having zero as eigenvalue of multiplicity one, with the corresponding eigenvector 
with no zero entries \cite{CoFoGo2017, ScFa2021, gutmansciriha}.
In a series of papers \cite{BaFoPiSc2022, circulantnuts1, Da2022A, circulantnuts2}, the problem of existence of nut graphs among various families of graphs
has been investigated.
In particular, the pairs $(n, d)$ for which there exists at least one $d$-regular nut graph of order $n$ have been determined up to degree $d \le 11$ in \cite{FoGaGoPiSc2020}, and for degree $d = 12$ in \cite{circulantnuts1}.
The methods used in these papers included mainly circulant graphs.
Based on these results, \cite{circulantnuts3, circulantnuts2} expanded this approach, with \cite{Da2022A} completely resolving the problem of existence of circulant nut graphs of a given order and degree.

Note that every circulant graph is a Cayley graph and, therefore, it is vertex-transitive.
It has been observed that each vertex-transitive nut graph has an even degree \cite{FoGaGoPiSc2020}.
One can conclude that the smallest degree for which a circulant nut graph exists is four.
The quartic circulant nut graphs were classified in \cite{Da2022B}, while the cubic tricirculant nut graphs were classified in \cite{BaDPZ}. For more results on the symmetries of nut graphs, see \cite{BaDam2024, BaDamFow2024, BaFow2024, BaFowPi2024, Damnjanovic2023_ARX}.

In this paper, the complete classification of quartic bicirculant nut graphs is given.
We first define bicirculant graphs.
Then we describe in more detail the connected quartic bicirculant graphs and give the statement of our main
theorem, which gives the classification of quartic bicirculant nut graphs.

A \emph{bicirculant} is a graph of order $n = 2m$ admitting an automorphism with two vertex orbits of equal
size $m$. It can be described as follows; see for example \cite{SRGbicirc}. Given an integer $m \ge 3$ and sets $S, T, R \subseteq \Z_m$ such that $S=-S$, $T=-T$, $R \ne \emptyset$ and $0 \not\in S \cup T$, 
the graph $\bicirc(m;S,T,R)$ has vertex set 
$V=\{x_0,\dots,x_{m-1},y_0,\dots,y_{m-1}\}$ and edge set 
\[
    E = \{x_ix_{i+j} \mid i \in\Z_m, j \in S\} \cup \{y_iy_{i+j} \mid i \in\Z_m, j \in T\} \cup \{x_iy_{i+j} \mid i \in\Z_m, j \in R\} .
\]

Bicirculants can also be described as regular $\Z_m$-covers over a graph on two vertices
with possible multiple edges, loops and semi-edges; see \cite{pisanskibicirc}.
From the theory of covering graphs, it follows that we may choose the elements of the sets $S, R, T$ to be such that $0 \in R$.
Moreover, in this case, a bicirculant graph is connected if and only if the elements from $S$, $T$ and $R$ generate the group $\Z_m$; see, for example, \cite{GT}.

The well-known generalized Petersen graphs are bicirculants of degree 3. More generally, an $I$-graph $I(m; j, k)$ is a bicirculant $\bicirc(m; S, T, R)$ with $m \ge 3$, $S=\{j, -j\}$, $T=\{k, -k\}$  and $R=\{ 0 \}$, where $1 \le j, k < \frac{m}{2}$. An $I$-graph is a
generalized Petersen graph if and only if $\gcd(m,j)=1$ or $\gcd(m,k)=1$.

In this paper we consider the connected quartic bicirculants. Such graphs fall into four classes with respect to the size of $R$, which can be 1, 2, 3 or 4; see \cite{quarticet}.
We denote these classes by $\cB_1$, $\cB_2$, $\cB_3$ and $\cB_4$.
A graph $\bicirc(m; S, T, R)$ belongs to
the class $\cB_i$ if $|R| = i$.

\begin{enumerate}

\item Class $\cB_1$:\ Here $m$ needs to be even, $m \ge 4$, $S = \{a, -a, \frac{m}{2} \}$, $T = \{b, -b, \frac{m}{2} \}$ and $R=\{0\}$.
For such parameters, we denote the graph $\bicirc(m; S, T, R)$ by $B_1(m;a,b)$.
Without loss of generality, we may assume that $1 \le a \le b < \frac{m}{2}$.

\item Class $\cB_2$:\ Here we have $m \ge 3$, $S = \{a, -a \}$, $T = \{b, -b \}$ and $R=\{0, c\}$. 
For such parameters, we denote the graph $\bicirc(m; S, T, R)$ by $B_2(m;a,b,c)$.
Without loss of generality, we may assume that $1 \le a \le b < \frac{m}{2}$ and $1 \le c \le \frac{m}{2}$.

\item Class $\cB_3$:\ Here $m$ needs to be even, $m \ge 4$, $S = \{ \frac{m}{2} \}$, $T=\{ \frac{m}{2} \}$ and $R = \{0, a, b\}$. 
For such parameters, we denote the graph $\bicirc(m; S, T, R)$ by $B_3(m;a,b)$.
Observe that $B_3(m; a, b)$ is isomorphic to both $B_3(m; -a, b - a)$ and $B_3(m; -b, a - b)$.
Thus, we can assume without loss of generality that $a$ and $b$ are of the same parity and $1 \le a < b < m$.

\item Class $\cB_4$:\ Here we have $m \ge 4$, $S = \emptyset$, $T = \emptyset$ and $R = \{0, a, b, c\}$. 
For such parameters, we denote the graph $\bicirc(m; S, T, R)$ by $B_4(m; a, b, c)$.
Without loss of generality, we may assume that $1 \le a < b < c < m$. 
\end{enumerate}

The four classes of connected quartic bicirculants can be viewed as regular covers over the four possible $\mathbb{Z}_m$-voltage graphs from Figure~\ref{voltages}.
Note that a semi-edge in the base graph can occur only when $m$ is even and it must have the voltage $\frac{m}{2}$.

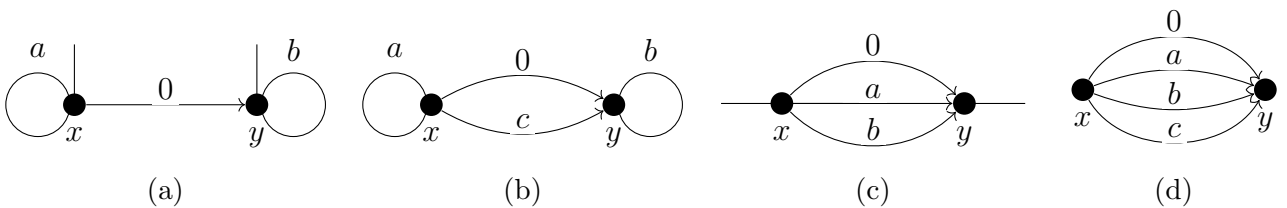
\begin{figure}[htbp]
\centering
\begin{subfigure}[b]{0.25\textwidth}
\centering
\begin{tikzpicture}[scale=0.8]
\def\x{3}
\draw (-0.6,0) circle (15pt); 
\draw (3.6,0) circle (15pt);
\node[ circle, fill= black, inner sep = 3pt, label=below:$x$] (x) at (0,0)  { };
\node[ circle, fill=black, inner sep = 3pt, label=below:$y$ ] (y) at (\x, 0) { };
\draw [->] (x) to node[above, inner sep = 2, fill = white] {$0$} (y);
\draw (0,1) -- (0,0);
\draw (3,1) -- (3,0);
\draw (x) + (-0.6,0.9) node {$a$};
\draw (y) + (0.6, 0.9) node {$b$};
\end{tikzpicture}
\caption{}
\end{subfigure}
\hspace{0.2cm}
\begin{subfigure}[b]{0.25\textwidth}
\centering
\begin{tikzpicture}[scale=0.8]
\def\x{3}
\draw (-0.6,0) circle (15pt); 
\draw (3.6,0) circle (15pt);
\node[ circle, fill= black, inner sep = 3pt, label=below:$x$] (x) at (0,0)  { };
\node[ circle, fill=black, inner sep = 3pt, label=below:$y$ ] (y) at (\x, 0) { };
\draw [->] (x) to[bend left = 30]  node[above, inner sep = 2, fill = white] {$0$}(y);
\draw [->] (x) to[bend right = 30]  node[above, inner sep = 2, fill = white] {$c$} (y);
\draw (x) + (-0.6,0.9) node {$a$};
\draw (y) + (0.6, 0.9) node {$b$};
\end{tikzpicture}
\caption{}
\end{subfigure}
\hspace{0.2cm}
\begin{subfigure}[b]{0.24\textwidth}
\centering
\begin{tikzpicture}[scale=0.8]
\def\x{3}
\node[ circle, fill= black, inner sep = 3pt, label=below:$x$] (x) at (0,0)  { };
\node[ circle, fill=black, inner sep = 3pt, label=below:$y$ ] (y) at (\x, 0) { };
\draw [->] (x) to[bend left = 45]  node[above, inner sep = 2, fill = white] {$0$}(y);
\draw [->] (x) to   node[above, inner sep = 2, fill = white] {$a$} (y);
\draw [->] (x) to[bend right = 45]  node[above, inner sep = 2, fill = white] {$b$} (y);
\draw (-1,0) --  (1,0);
\draw ( 3,0) --  (4,0);
\end{tikzpicture}
\caption{}
\end{subfigure}
\hspace{0.05cm}
\begin{subfigure}[b]{0.19\textwidth}
\centering
\begin{tikzpicture}[scale=0.8]
\def\x{3}
\node[ circle, fill= black, inner sep = 3pt, label=below:$x$] (x) at (0,0)  { };
\node[ circle, fill=black, inner sep = 3pt, label=below:$y$ ] (y) at (\x, 0) { };
\draw [->] (x) to[bend left = 60]  node[above, inner sep = 2, fill = white] {$0$}(y);
\draw [->] (x) to[bend left = 20]  node[above, inner sep = 2, fill = white] {$a$} (y);
\draw [->] (x) to[bend right = 20]  node[above, inner sep = 2, fill = white] {$b$} (y);
\draw [->] (x) to[bend right =60]  node[above, inner sep = 2, fill = white] {$c$} (y);
\end{tikzpicture}
\caption{}
\end{subfigure}

\caption{Possible $\mathbb{Z}_m$-voltage graphs for the connected quartic bicirculants. Cases (1) and (3) occur only when $m$ is even and in this case, the voltages on the semi-edges are equal to $\frac{m}{2}$.}
\label{voltages}
\end{figure}

Note that these four classes are not disjoint. 
For example, $B_2(4; 1, 1, 2) \in \cB_2$ is isomorphic to 
$B_4(4; 1, 2, 3) \in \cB_4$, and $B_2(6;1,1,3) \in \cB_2$ is isomorphic to $B_3(6;1,5) \in \cB_3$; see \cite{quarticet}.

The graphs from $\cB_4$ (with $|R| = 4$) are known as the \emph{cyclic Haar graphs} \cite{haar}.
The family of Rose Window graphs was introduced by Steve Wilson in \cite{wilsonRW}. Given integers $m \ge 3$, $1 \le a < \frac{m}{2}$ and $1 \le r < m$, the \emph{Rose Window graph} $R_m(a, r)$ is equal to $B_2(m; 1, a, m - r)$. For this reason,
we call the bicirculants from $\cB_2$ (with $|R|=2$) the \emph{generalized Rose Window graphs}.

We now state our main theorem. For this purpose, we need to introduce the following notation.
For every $x \in \mathbb{Z}, \, x \neq 0$, let $v_p(x)$ denote the power of a prime $p$ in the prime factorization of $|x|$, i.e., the unique nonnegative integer such that $p^{v_p(x)}$ divides $x$, but $p^{v_p(x)+1}$ does not divide $x$.

\begin{theorem}[Quartic bicirculant nut graph classification]\label{main_theorem}
A quartic bicirculant graph is a nut graph if and only if it is isomorphic to one of the graphs below:
\begin{enumerate}
\item the graph $B_1(m;a,b)$, where $m$ is even, $m \ge 4$, $1 \le a \le b < \frac{m}{2}$, and the following conditions hold:
\begin{enumerate}[(i)]
\item $m \equiv_4 2$;
\item $a$ and $b$ are both even; 
\item $\gcd(\frac{m}{2}, a, b) = 1$; 
\item if $5 \mid m$, then at least one of the numbers $a, b, a - b, a + b$ is divisible by five;
\end{enumerate}

\item the graph $B_2(m;a,b,c)$, where $m \ge 3$, $1 \le a \le b < \frac{m}{2}$, $1 \le c \le \frac{m}{2}$, and the following conditions hold:
\begin{enumerate}[(i)]
    \item $m$ is coprime to all of $\gcd(a - b, a + b + c)$, $\gcd(a - b, a + b - c)$, $\gcd(a + b, a - b + c)$ and $\gcd(a + b, a - b - c)$;
    \item if $v_2(m) > v_2(c)$, then neither $v_2(a)$ nor $v_2(b)$ are equal to $v_2(c) - 1$;
    \item if $12 \mid m$, then $(a + b, a - b, c) \not\equiv_{12} (\pm 2, \pm 2, \pm 3)$;
    \item if $30 \mid m$, then $(\{a + b, a - b\}, c) \not\equiv_{30} (\{\pm 3, \pm 5\}, \pm 6), (\{\pm 3, \pm 9\}, \pm 10), (\{\pm 5, \pm 9\}, \linebreak \pm 12)$;
\end{enumerate}

\item the graph $B_3(m;a,b)$, where $m$ is even, $m \ge 4$,  $1 \le a < b < m$, and $a$ and $b$ are both odd, while $\gcd(m, a) = \gcd(m, b) = 1$ and $v_2(b - a) \ge v_2(m)$.
\end{enumerate}
In particular, no graphs from the class $\cB_4$ are nut graphs. In conditions 2(iii) and 2(iv), the $\pm$ signs can be chosen independently.
\end{theorem}

In the rest of the paper we focus on proving Theorem \ref{main_theorem}.
Since bipartite graphs are never nut graphs (see Subsection~\ref{sec:nuts}), none of the graphs from the class $\cB_4$ can be a nut graph.
We thus consider each of the classes $\cB_1$, $\cB_2$ and $\cB_3$ in a separate section.

\section{Preliminaries}\label{sc_prel}

In this section we describe the necessary definitions and tools regarding the nut graphs, the eigenvalues of bicirculants and the cyclotomic polynomials.

\subsection{Basic results about nut graphs}\label{sec:nuts}

In the present subsection we list the main properties of nut graphs that we will need. It is well known and easy to see that nut graphs are connected and nonbipartite \cite{gutmansciriha}.

\begin{lemma}\label{lemma_conn_bip}
Every nut graph is connected. A bipartite graph is not a nut graph.
\end{lemma}

Eigenvectors corresponding to the eigenvalue zero are also called \emph{kernel vectors}. The following lemma characterizes the kernel vectors of a given graph.

\begin{lemma}\label{local_condition_lemma}
Let $G$ be a graph and let $u$ be a nonzero vector.  
Then $u$ is a kernel eigenvector if and only if for every vertex $x$ of $G$ the following   holds:
\[
    \sum_{y \sim x} u(y) = 0.
\]
\end{lemma}
\begin{proof}
Multiply the row of the adjacency matrix of $G$, corresponding to the vertex $x$, by the eigenvector $u$, and the result is obtained.
\end{proof} 

We call the above condition that a kernel eigenvector must satisfy the \emph{local condition} \cite{FoGaGoPiSc2020}. We now state a lemma that gives a connection between the symmetries of a graph and
the structure of the null space vectors of a nut graph. More generally, the lemma gives a connection
between the symmetries of a graph and structure of the eigenspace corresponding to a simple eigenvalue.
For the proof, see \cite[p. 135]{Cvetkovic1995}.
\begin{lemma} \label{nut_orbit_lemma}
Let $G$ be a graph and let $\pi \in \mathrm{Aut}(G)$ be its automorphism. 
Let $u$ be an eigenvector corresponding to a simple eigenvalue of $G$.
If $X = \{ x_0, x_1, \ldots, x_{k-1} \} \subseteq V(G)$ 
represents an orbit of $\pi$ such that
\[
    \pi(x_j) = x_{j+1} \qquad (j = 0, \dots, k-1),
\]
where addition is done modulo $k$, then we have that $u$ is constant on $X$, or the orbit size $k$ is even and
\[
    u(x_j) = (-1)^j \, u(x_0) \qquad  (j =0, \dots, k-1).
\]
\end{lemma}

We now turn our attention to the quartic bicirculants.
The graphs from the class $\cB_4$, the cyclic Haar graphs, are obviously all bipartite. Therefore, none of them are nut graphs. We will now show that the other connected quartic bicirculant graphs are nut graphs if and only if they have a simple eigenvalue $0$.

\begin{lemma}  \label{simple_nut}
Let $G$ be a connected quartic bicirculant graph with a simple eigenvalue $0$.
Then the eigenvector corresponding to the eigenvalue $0$ has no zero entries.
\end{lemma}
\begin{proof}
Denote by $u$ the eigenvector corresponding to the eigenvalue $0$.
Since $G$ is a bicirculant, it has an automorphism $\pi$ with two vertex orbits of equal size, say $\{x_0, \ldots, x_{m-1}\}$ and $\{y_0, \ldots, y_{m-1}\}$. By Lemma \ref{nut_orbit_lemma}, the components of $u$ corresponding
to the vertices in the same orbit of $\pi$ have the same absolute value.
So if $u$ has a zero entry, then all of the components corresponding to the vertices of one orbit of $\pi$ have to be zero, while the other components are nonzero, with the 
same absolute value. If $G$ is from the class $\cB_1$ or $\cB_3$, the local condition implies that this is not possible (an odd number of numbers with the same absolute value cannot sum up to zero).

If $G$ is from the class $\cB_2$, the eigenvalue $0$ has a nonzero eigenvector $u$ as follows:\ set the components corresponding to $x_i$ to $1$ and the components corresponding to $y_i$ to $-1$. Therefore, no eigenvector corresponding to the (simple) eigenvalue $0$ can have a zero entry.

Finally, if $G$ is from the class $\cB_4$, then $0$ is not a simple eigenvalue, since $G$ is bipartite.
\end{proof}

\subsection{Eigenvalues of bicirculants}

In this subsection we describe the eigenvalues of   bicirculant graphs. The eigenvalues of bi-Cayley graphs, and in particular, of bicirculant graphs, are known \cite{nCayleyeigenvalues}.
For a positive integer $m$, let $\omega_m$ denote the $m$-th root of unity $e^{2\pi i/m}$.

\begin{theorem}{\rm (\cite[Theorem 3.2]{nCayleyeigenvalues})}   \label{thm:bicirceigen}
The graph $\bicirc(m;S,T,R)$ has eigenvalues
\[
    \frac{\lambda_k^S+\lambda_k^T \pm \sqrt{(\lambda_k^S-\lambda_k^T)^2 +4|\lambda_k^R|^2}}{2}, \qquad k=0,\dots,m-1,
\]
where 
\[
    \lambda_k^S= \sum_{s \in S} \, \omega_m^{ks}, \qquad \lambda_k^T= \sum_{t \in T} \, \omega_m^{kt} \qquad \mbox{and} \qquad \lambda_k^R= \sum_{r \in R} \, \omega_m^{kr}.
\]
\end{theorem}

\begin{remark}
The numbers $\lambda_k^S$, $\lambda_k^T$ and $\lambda_k^R$ from Theorem \ref{thm:bicirceigen} are the eigenvalues of the circulant graphs with connection sets $S$ and $T$ and the circulant digraph with connection set $R$, respectively.
\end{remark}

\begin{corollary} \label{zeroeigenvalue}
Any graph $\bicirc(m;S,T,R)$ from classes $\cB_1$, $\cB_2$ or $\cB_3$ has an eigenvalue $0$ if and only if for some $k \in \{0, \ldots, m-1\}$ the following equality holds:
\begin{equation}\label{general_eigen}
    |\lambda_k^R|^2 =\lambda_k^S \cdot \lambda_k^T,
\end{equation}
where 
\[
    \lambda_k^S= \sum_{s \in S} \, \omega_m^{ks}, \qquad \lambda_k^T= \sum_{t \in T} \, \omega_m^{kt} \qquad \mbox{and} \qquad \lambda_k^R= \sum_{r \in R} \, \omega_m^{kr}.
\]
Moreover, $0$ is a simple eigenvalue if and only if the equality in \eqref{general_eigen} holds for exactly one $k$.
\end{corollary}
\begin{proof}
Suppose that the graph $\bicirc(m; S, T, R)$ has an eigenvalue $0$. Let $k$ be an integer such that 
\[
    \frac{\lambda_k^S+\lambda_k^T \pm \sqrt{(\lambda_k^S-\lambda_k^T)^2 +4|\lambda_k^R|^2}}{2} = 0.
\]
Such an integer exists by Theorem \ref{thm:bicirceigen}. Then  $(\lambda_k^S+\lambda_k^T)^2=(\lambda_k^S-\lambda_k^T)^2 +4|\lambda_k^R|^2$, whence it follows that  $|\lambda_k^R|^2 = \lambda_k^S \cdot \lambda_k^T$. Conversely, if an integer $k$ satisfies \eqref{general_eigen}, it is easy to see that the same $k$ corresponds to the eigenvalue $0$ in  Theorem \ref{thm:bicirceigen}.

If there is more than one $k$ for which \eqref{general_eigen} holds, then Theorem \ref{thm:bicirceigen} implies that $0$ is not a simple eigenvalue. If there is exactly one $k$ for which \eqref{general_eigen} holds, then $k = 0$ or $k = \frac{m}{2}$, since otherwise, \eqref{general_eigen} would also hold for $k' = m - k$. In this case, $0$ is a simple eigenvalue. Indeed, since $\lambda_k^S$ and $ \lambda_k^T$ are real numbers, the only way for $0$ to be a double eigenvalue is when $\lambda_k^S = \lambda_k^T = \lambda_k^R = 0$. However, this is impossible for a graph of class $\cB_1$ or $\cB_3$, since the values $\lambda_0^S$ and $\lambda_{m / 2}^S$ are odd. On the other hand, for a graph of class $\cB_2$, the values $\lambda_0^S$ and $\lambda_{m / 2}^S$ are either $2$ or $-2$.
\end{proof}

\subsection{Cyclotomic polynomials}

For any $f \in \mathbb{N}$, the \emph{cyclotomic polynomial} $\Phi_f(x)$ is defined as
\[
    \Phi_f(x) = \prod_{\zeta} (x - \zeta) ,
\]
where $\zeta$ ranges over the primitive $f$-th roots of unity. It is known that the cyclotomic polynomials have integer coefficients and are irreducible in $\mathbb{Q}[x]$ (see, for example, \cite[Chapter~33]{gallian}). Thus, any $P(x) \in \mathbb{Q}[x]$ contains a primitive $f$-th root of unity among its roots if and only if $\Phi_f(x) \mid P(x)$. Furthermore, if $p^2 \mid f$ for some prime $p$, then $\Phi_f(x) = \Phi_{f / p}(x^p)$. We will frequently combine this fact with the following folklore lemma.

\begin{lemma}[\hspace{1sp}{\cite[Lemma 18]{BaDPZ}}]\label{cool_div_lemma}
    Let $V(x), W(x) \in \mathbb{Q}[x], \, W(x) \not\equiv 0$, be such that $W(x) \mid V(x)$ and the powers of all the nonzero terms of $W(x)$ are divisible by $\beta \in \mathbb{N}$. Also, for any $0 \le j < \beta$, let $V^{(\beta, j)}(x)$ be the polynomial comprising the terms of $V(x)$ whose power is congruent to $j$ modulo $\beta$. Then $W(x) \mid V^{(\beta, j)}(x)$ for every $0 \le j < \beta$.
\end{lemma}

We will also need the next theorem by Filaseta and Schinzel on the divisibility of lacunary polynomials by cyclotomic polynomials.

\begin{theorem}[\hspace{1sp}{\cite[Theorem 2]{fischinz2003}}]\label{filaschinz_th}
Let $P(x) \in \mathbb{Z}[x]$ have $N$ nonzero terms and let $\Phi_f(x) \mid P(x)$. Suppose that $p_1, p_2, \dots, p_k$ are distinct primes such that
\[
    \sum_{j=1}^k (p_j-2) > N-2 .
\]
Let $e_j = v_{p_j}(f)$. Then for at least one $j$, $1 \le j \le k$, we have that $\Phi_{f'}(x)\mid P(x)$, where $f' = f / p_j^{e_j}$.
\end{theorem}

\section{The class \texorpdfstring{$\cB_1$}{B1}}

In this section we consider the class $\cB_1$. Recall that these graphs are of the form $\bicirc(m; S, T, \linebreak R)$ with even $m \ge 4$, $S = \{a, -a, \frac{m}{2} \}$, $T = \{b, -b, \frac{m}{2}\}$ and $R = \{ 0 \}$, where $1 \le a \le b < \frac{m}{2}$. For such parameters, we denoted the graph $\bicirc(m;S,T,R)$ by $B_1(m;a,b)$. The following theorem provides a classification of nut graphs among the members of $\cB_1$.

\begin{theorem}\label{b1_main_th}
Let $m \ge 4$ be even and let $a$ and $b$ be integers such that $1 \le a \le b < \frac{m}{2}$. Then the graph $B_1(m; a, b)$ is a nut graph if and only if the following conditions hold:
\begin{enumerate}[(i)]
\item $m \equiv_4 2$; 
\item $a$ and $b$ are both even; 
\item $\gcd(\frac{m}{2}, a, b) = 1$; 
\item if $5 \mid m$, then at least one of the numbers $a, b, a - b, a + b$ is divisible by five.
\end{enumerate}
\end{theorem}

The goal of the present section is to give the proof of Theorem \ref{b1_main_th}. We start by observing the next result.

\begin{lemma}\label{b1_poly_lemma}
    The graph $B_1(m; a, b)$ is a nut graph if and only if $m \equiv_4 2$, $a$ and $b$ are both even, while the polynomial
    \begin{equation}\label{b1_aux_0}
        R_{a, b}(x) = x^{2a + 2b} + x^{2a} + x^{2b} + 1 + x^{2a + b} + x^{a + 2b} + x^a + x^b    
    \end{equation}
    is not divisible by $\Phi_f(x)$ for any odd $f \ge 3$ such that $f \mid m$, and the polynomial
    \begin{equation}\label{b1_aux_1}
        Q_{a, b}(x) = x^{2a + 2b} + x^{2a} + x^{2b} + 1 - x^{2a + b} - x^{a + 2b} - x^a - x^b    
    \end{equation}
    is not divisible by $\Phi_f(x)$ for any even $f \ge 4$ such that $f \mid m$.
\end{lemma}
\begin{proof}
    By Lemma~\ref{simple_nut} and Corollary~\ref{zeroeigenvalue}, we have that $B_1(m; a, b)$ is a nut graph if and only if
    \begin{equation}\label{b1_aux_2}
        (\zeta^a + \zeta^{-a} + \zeta^\frac{m}{2})(\zeta^b + \zeta^{-b} + \zeta^\frac{m}{2}) = |1|^2
    \end{equation}
    holds for exactly one $m$-th root of unity $\zeta$. Note that \eqref{b1_aux_2} is not satisfied for $\zeta = 1$, and if \eqref{b1_aux_2} holds for a nonreal $\zeta$, then it also holds for $\overline{\zeta} \neq \zeta$. With this in mind, we conclude that $B_1(m; a, b)$ is a nut graph if and only if $P_{a, b}(\zeta) = 0$ holds only for $\zeta = -1$ among all the $m$-th roots of unity $\zeta$, where
    \[
        P_{a, b}(\zeta) = (\zeta^a + \zeta^{-a})(\zeta^b + \zeta^{-b}) + \zeta^\frac{m}{2}(\zeta^a + \zeta^{-a} + \zeta^b + \zeta^{-b}) .
    \]

    Suppose that $B_1(m; a, b)$ is a nut graph. Observe that if $a$ and $b$ are of different parity, then
    \[
        (-1)^a + (-1)^{-a} + (-1)^b + (-1)^{-b} = 0 ,
    \]
    hence $P_{a, b}(-1) = 4(-1)^{a + b} \neq 0$. Thus, we have $a \equiv_2 b$, which implies
    \[
        ((-1)^a + (-1)^{-a})((-1)^b + (-1)^{-b}) = 4.
    \]
    Now, suppose that $4 \mid m$. If $a$ and $b$ are both even, we get
    \[
        (-1)^\frac{m}{2} ((-1)^a + (-1)^{-a} + (-1)^b + (-1)^{-b}) = 1 \cdot 4 = 4,
    \]
    which means that $P_{a, b}(-1) = 8 \neq 0$, thus yielding a contradiction. On the other hand, if $a$ and $b$ are both odd, we have
    \[
        P_{a, b}(i) = (i^a + i^{-a})(i^b + i^{-b}) + i^\frac{m}{2} (i^a + i^{-a} + i^b + i^{-b}) = 0 \cdot 0 + i^\frac{m}{2} \cdot (0 + 0) = 0 .
    \]
    This implies that $P_{a, b}(\zeta) = 0$ holds for an $m$-th root of unity $\zeta \neq -1$, which again leads to a contradiction. With all of this in mind, we conclude that if $B_1(m; a, b)$ is a nut graph, then $m \equiv_4 2$. In this case, we have $(-1)^\frac{m}{2} = -1$, hence we also get
    \[
        (-1)^a + (-1)^{-a} + (-1)^b + (-1)^{-b} = 4,
    \]
    which implies that $a$ and $b$ are both even.

    For an $m$-th root of unity $\zeta$ whose order $f$ is odd, we have $f \mid \frac{m}{2}$, which means that $P_{a, b}(\zeta) = 0$ if and only if
    \begin{equation}\label{b1_aux_3}
        (\zeta^a + \zeta^{-a})(\zeta^b + \zeta^{-b}) + (\zeta^a + \zeta^{-a} + \zeta^b + \zeta^{-b}) = 0.
    \end{equation}
    Similarly, for an $m$-th root of unity $\zeta$ whose order $f$ is even, we have $\frac{m}{2} \equiv_f \frac{f}{2}$, hence $P_{a, b}(\zeta) = 0$ if and only if
    \begin{equation}\label{b1_aux_4}
        (\zeta^a + \zeta^{-a})(\zeta^b + \zeta^{-b}) - (\zeta^a + \zeta^{-a} + \zeta^b + \zeta^{-b}) = 0.
    \end{equation}
    The result now follows from the irreducibility of cyclotomic polynomials after multiplying \eqref{b1_aux_3} and \eqref{b1_aux_4} by $\zeta^{a + b}$.
\end{proof}

Throughout the rest of the section, let $Q_{a, b}(x)$ and $R_{a, b}(x)$ be the polynomials~\eqref{b1_aux_1} and \eqref{b1_aux_0} from Lemma~\ref{b1_poly_lemma}, respectively. We continue by showing that conditions (i)--(iv) from Theorem~\ref{b1_main_th} are all necessary.

\begin{lemma}\label{b1_cond_lemma}
    If at least one of the conditions (i)--(iv) from Theorem \ref{b1_main_th} does not hold, then the graph $B_1(m; a, b)$ is not a nut graph.
\end{lemma}
\begin{proof}
    As shown in Lemma \ref{b1_poly_lemma}, if one of the conditions (i) and (ii) does not hold, then $B_1(n; a, b)$ is not a nut graph. Suppose that conditions (i) and (ii) both hold. If condition (iii) is not satisfied, then $\frac{m}{2}$, $a$ and $b$ all share a common odd prime factor $p$. By letting $\zeta$ be any primitive $2p$-th root of unity, we get
    \[
        Q_{a, b}(\zeta) = \zeta^{2a + 2b} + \zeta^{2a} + \zeta^{2b} + 1 - \zeta^{2a + b} - \zeta^{a + 2b} - \zeta^a - \zeta^b = 0.
    \]
    Since $2p \mid m$, Lemma \ref{b1_poly_lemma} implies that $B_1(m; a, b)$ is not a nut graph.

    Now, suppose that condition (iv) does not hold, so that $5 \mid m$ and $5 \nmid a, b, a - b, a + b$. In this case, we trivially observe that the numbers $2a + 2b, 2a, 2b, 0, a + b$ all have mutually distinct remainders modulo five. Since all of them are even, we obtain
    \begin{equation}\label{b1_aux_5}
        x^{2a + 2b} + x^{2a} + x^{2b} + 1 + x^{a + b} \equiv_{\Phi_{10}(x)} x^8 + x^6 + x^4 + x^2 + 1 .
    \end{equation}
    Similarly, the numbers $2a + b, a + 2b, a, b, a + b$ all have mutually distinct remainders modulo five, hence we have
    \begin{equation}\label{b1_aux_6}
        x^{2a + b} + x^{a + 2b} + x^{a} + x^b + x^{a + b} \equiv_{\Phi_{10}(x)} x^8 + x^6 + x^4 + x^2 + 1 .
    \end{equation}
    From \eqref{b1_aux_5} and \eqref{b1_aux_6}, we get $\Phi_{10}(x) \mid Q_{a, b}(x)$. Since $10 \mid m$, Lemma \ref{b1_poly_lemma} implies that $B_1(m; a, b)$ is not a nut graph.
\end{proof}

In the rest of the section, we finalize the proof of Theorem \ref{b1_main_th} by showing that conditions~(i)--(iv) are also sufficient. We note that the divisibility of $Q_{a, b}(x)$ and $R_{a, b}(x)$ polynomials by cyclotomic polynomials has already been investigated in \cite[Section 6]{BaDPZ}. By relying on these existing results, we obtain the next two lemmas.

\begin{lemma}\label{b1_square_lemma}
    Suppose that conditions (i)--(iv) from Theorem \ref{b1_main_th} hold and let $f \in \mathbb{N}$ be such that $f \mid m$. If $\Phi_f(x) \mid Q_{a, b}(x)$ or $\Phi_f(x) \mid R_{a, b}(x)$, then $f$ is square-free.
\end{lemma}
\begin{proof}
    By way of contradiction, suppose that $\Phi_f(x) \mid Q_{a, b}(x)$ or $\Phi_f(x) \mid R_{a, b}(x)$ holds for some $f \in \mathbb{N}$ such that $f \mid m$ and $p^2 \mid f$, where $p$ is a prime. Since $m \equiv_4 2$, we get that $p \ge 3$. Also, note that $\Phi_f(x) = \Phi_{f / p}(x^p)$.
    
    If $p = 3$, then it follows that $3 \mid a, b$, as shown in \cite[Lemma 21]{BaDPZ}, which contradicts condition~(iii). On the other hand, if $p = 5$, then we get $5 \mid a, b$ or $5 \nmid a, b, a + b, a - b$, as shown in \cite[Lemma 20]{BaDPZ}. Therefore, at least one of the conditions (iii) and (iv) does not hold, which is impossible. Finally, if $p \ge 7$, then we may conclude that $p \mid a, b$, as shown in \cite[Lemma~19]{BaDPZ}, which contradicts condition~(iii).
\end{proof}

\begin{lemma}\label{b1_prime_lemma}
    Suppose that conditions (i)--(iv) from Theorem \ref{b1_main_th} hold and let $p \ge 11$ be a prime such that $p \mid m$. Then $\Phi_p(x) \nmid R_{a, b}(x)$ and $\Phi_{2p}(x) \nmid Q_{a, b}(x)$.
\end{lemma}
\begin{proof}
    First of all, suppose that $\Phi_p(x) \mid R_{a, b}(x)$. As shown in \cite[Lemma 23]{BaDPZ}, this implies that $p \mid a, b$, which contradicts condition~(iii) from Theorem~\ref{b1_main_th}.

    Now, suppose that $\Phi_{2p}(x) \mid Q_{a, b}(x)$. We trivially observe that the auxiliary polynomial
    \begin{align*}
        Q_{a, b}^{\bmod 2p}(x) &= (-1)^{\lfloor \frac{2a + 2b}{p} \rfloor} x^{(2a + 2b) \bmod p} + (-1)^{\lfloor \frac{2a}{p} \rfloor} x^{2a \bmod p} + (-1)^{\lfloor \frac{2b}{p} \rfloor} x^{2b \bmod p} + 1\\
        &-(-1)^{\lfloor \frac{2a + b}{p} \rfloor} x^{(2a + b) \bmod p} - (-1)^{\lfloor \frac{a + 2b}{p} \rfloor} x^{(a + 2b) \bmod p}\\
        &- (-1)^{\lfloor \frac{a}{p} \rfloor} x^{a \bmod p} - (-1)^{\lfloor \frac{b}{p} \rfloor} x^{b \bmod p}
    \end{align*}
    is also divisible by $\Phi_{2p}(x)$. Observe that $\Phi_{2p}(x) = \sum_{j = 0}^{p - 1} (-1)^j x^j$. Since $\deg Q_{a, b}^{\bmod 2p} \le p - 1 = \deg \Phi_{2p}$, we have that either $Q_{a, b}^{\bmod 2p}(x) \equiv 0$ or $Q_{a, b}^{\bmod 2p}(x) = \beta \, \Phi_{2p}(x)$ for some $\beta \in \mathbb{Q} \setminus \{ 0 \}$. In the latter case, $Q_{a, b}^{\bmod 2p}(x)$ must have $p$ nonzero terms, which is not possible since $p \ge 11$. Thus, $Q_{a, b}^{\bmod 2p}(x) \equiv 0$. The same approach from \cite[Lemma 23]{BaDPZ} can now be used to verify that this is only possible when $p \mid a, b$, which contradicts condition~(iii).
\end{proof}

We now apply Lemmas~\ref{b1_poly_lemma}, \ref{b1_square_lemma} and \ref{b1_prime_lemma} alongside Theorem \ref{filaschinz_th} to obtain the following result.

\begin{lemma}\label{b1_key_lemma}
    Suppose that conditions (i)--(iv) from Theorem \ref{b1_main_th} hold. Then the graph $B_1(m; a, \linebreak b)$ is a nut graph if and only if $\Phi_f(x) \nmid R_{a, b}(x)$ for every $f \in \{3, 5, 7, 15, 21\}$ such that $f \mid m$, and $\Phi_f(x) \nmid Q_{a, b}(x)$ for every $f \in \{ 6, 10, 14, 30, 42 \}$ such that $f \mid m$.
\end{lemma}
\begin{proof}
    Let $F_1 = \{3, 5, 7, 15, 21\}$ and $F_2 = \{6, 10, 14, 30, 42\}$. If $\Phi_f(x) \mid R_{a, b}(x)$ for some $f \in F_1$ such that $f \mid m$, or $\Phi_f(x) \mid Q_{a, b}(x)$ for some $f \in F_2$ such that $f \mid m$, then Lemma \ref{b1_poly_lemma} implies that $B_1(m; a, b)$ is not a nut graph. Thus, it suffices to prove the converse.

    Suppose that $B_1(m; a, b)$ is not a nut graph. By Lemma \ref{b1_poly_lemma}, there exists an $f \ge 3$ such that $f \mid m$ and one of the next two conditions is satisfied:
    \begin{enumerate}[(a)]
        \item $f$ is odd and $\Phi_f(x) \mid R_{a, b}(x)$;
        \item $f$ is even and $\Phi_f(x) \mid Q_{a, b}(x)$.
    \end{enumerate}
    From Lemma \ref{b1_square_lemma}, we get that $f$ is square-free. We now divide the problem into two cases depending on whether $f$ has a prime factor among $3$, $5$ and $7$.

    \bigskip\noindent
    \textbf{Case 1:} $\gcd(f, 3 \cdot 5 \cdot 7) = 1$.\quad
    In this case, Theorem \ref{filaschinz_th} can be repeatedly used to cancel out prime factors of $f$ greater than $7$ until only one such factor is left. Therefore, depending on whether the starting $f$ is odd or even, we obtain either $\Phi_p(x) \mid R_{a, b}(x)$ for some prime $p \ge 11$ such that $p \mid m$, or $\Phi_{2p}(x) \mid Q_{a, b}(x)$ for some prime $p \ge 11$ such that $p \mid m$. Either way, Lemma~\ref{b1_prime_lemma} yields a contradiction.

    \bigskip\noindent
    \textbf{Case 2:} $\gcd(f, 3 \cdot 5 \cdot 7) > 1$.\quad
    Here, Theorem \ref{filaschinz_th} can be repeatedly applied to cancel out any potential prime factors of $f$ greater than $7$. Therefore, there exists a square-free $f' \ge 3$ without a prime factor above $7$ such that $f' \mid f$ and one of the next two conditions is satisfied:
    \begin{enumerate}[(a)]
        \item $f'$ is odd and $\Phi_{f'}(x) \mid R_{a, b}(x)$;
        \item $f'$ is even and $\Phi_{f'}(x) \mid Q_{a, b}(x)$.
    \end{enumerate}
    Also, since $(5 - 2) + (7 - 2) > 8 - 2$, we can cancel out either $5$ or $7$ as a prime factor if they appear as simultaneous factors. Hence, we may assume without loss of generality that $35 \nmid f'$. The result now follows by observing that either $f' \in F_1$ and $\Phi_{f'}(x) \mid R_{a, b}(x)$, or $f' \in F_2$ and $\Phi_{f'}(x) \mid Q_{a, b}(x)$.
\end{proof}

We are finally in a position to complete the proof of Theorem \ref{b1_main_th} via Lemmas \ref{b1_cond_lemma} and \ref{b1_key_lemma}.

\begin{proof}[Proof of Theorem \ref{b1_main_th}]
    If at least one of the conditions (i)--(iv) from Theorem \ref{b1_main_th} does not hold, then $B_1(m; a, b)$ is indeed not a nut graph, due to Lemma \ref{b1_cond_lemma}. Now, suppose that all of these conditions hold. For each $f \in \mathbb{N}$, let
    \begin{align*}
        Q_{a, b}^{\bmod f}(x) = x^{(2a + 2b) \bmod f} &+ x^{2a \bmod f} + x^{2b \bmod f} + 1\\
        &- x^{(2a + b) \bmod f} - x^{(a + 2b) \bmod f} - x^{a \bmod f} - x^{b \bmod f}
    \end{align*}
    and
    \begin{align*}
        R_{a, b}^{\bmod f}(x) = x^{(2a + 2b) \bmod f} &+ x^{2a \bmod f} + x^{2b \bmod f} + 1\\
        &+ x^{(2a + b) \bmod f} + x^{(a + 2b) \bmod f} + x^{a \bmod f} + x^{b \bmod f} .
    \end{align*}
    We trivially observe that $\Phi_f(x) \mid Q_{a, b}(x)$ if and only if $\Phi_f(x) \mid Q_{a, b}^{\bmod f}(x)$, and $\Phi_f(x) \mid R_{a, b}(x)$ if and only if $\Phi_f(x) \mid R_{a, b}^{\bmod f}(x)$. As indicated by the computational results given in \cite[Sections~A and B]{BaDPZ}, provided conditions (i)--(iv) are satisfied, $\Phi_f(x) \mid R_{a, b}^{\bmod f}(x)$ cannot be true for any $f \in \{ 3, 5, 7, 15, 21 \}$, while
    $\Phi_f(x) \mid Q_{a, b}^{\bmod f}(x)$ cannot be true for any $f \in \{ 6, 10, 14, 30, 42 \}$ such that $f \mid m$. By Lemma \ref{b1_key_lemma}, we conclude that the graph $B_1(m; a, b)$ is a nut graph.
\end{proof}

We end the section with the following remark.

\begin{remark}
Let $m \ge 6$ be such that $m \equiv_4 2$ and let $a, b$ be even integers such that $1 \le a \le b < \frac{m}{2}$. Then the graph $B_1(m; a, b)$ is isomorphic to the product $I(\frac{m}{2}, \frac{a}{2}, \frac{b}{2}) \sqr K_2$. Therefore, the nut graphs from the class $\cB_1$ correspond to the $I$-graphs that have either $1$ or $-1$ as a simple eigenvalue, since the eigenvalues of $K_2$ are $-1$ and $1$. The generalized Petersen graphs that have $1$ as a simple eigenvalue were classified by Guo and Mohar \cite{GuoMoharsimpleGP}.

\end{remark}

\section{The class \texorpdfstring{$\cB_2$}{B2} --- The generalized rose window graphs}
\label{section:B2_new}

In this section we consider the class $\cB_2$. Recall that these graphs are of the form $\bicirc(m; S, T, \linebreak R)$ with $m \ge 3$, $S = \{a,-a\}$, $T = \{b,-b\}$ and $R = \{0, c\}$, where $1 \le a \le b < \frac{m}{2}$ and $1 \le c \le \frac{m}{2}$. For such parameters, we denoted the graph $\bicirc(m; S, T, R)$ by $B_2(m;a,b,c)$. The following theorem provides a classification of nut graphs among the members of $\cB_2$.

\begin{theorem}\label{b2_main_th}
    Let $m \ge 3$ and let $a$, $b$ and $c$ be integers such that $1 \le a \le b < \frac{m}{2}$ and $1 \le c \le \frac{m}{2}$. Then the graph $B_2(m;a,b,c)$ is a nut graph if and only if the following conditions hold:
    \begin{enumerate}[(i)]
        \item $m$ is coprime to all of $\gcd(a - b, a + b + c)$, $\gcd(a - b, a + b - c)$, $\gcd(a + b, a - b + c)$ and $\gcd(a + b, a - b - c)$;
        \item if $v_2(m) > v_2(c)$, then neither $v_2(a)$ nor $v_2(b)$ are equal to $v_2(c) - 1$;
        \item if $12 \mid m$, then $(a + b, a - b, c) \not\equiv_{12} (\pm 2, \pm 2, \pm 3)$;
        \item if $30 \mid m$, then $(\{a + b, a - b\}, c) \not\equiv_{30} (\{\pm 3, \pm 5\}, \pm 6), (\{\pm 3, \pm 9\}, \pm 10), (\{\pm 5, \pm 9\}, \pm 12)$.
    \end{enumerate}
    In conditions (iii) and (iv), the $\pm$ signs can be chosen independently. In particular, there are $8$ and $48$ forbidden triples of remainders in conditions (iii) and (iv), respectively.
\end{theorem}

\begin{remark}
    The members of $\cB_2$ for which $a = 1$ comprise the Rose Window graphs \cite{wilsonRW}. Hence, Theorem \ref{b2_main_th} also provides a classification of nut graphs among these graphs.
\end{remark}

The goal of the present section is to give the proof of Theorem \ref{b2_main_th}. We begin by observing the following result.

\begin{lemma}\label{b2_poly_lemma}
    The graph $B_2(m; a, b, c)$ is a nut graph if and only if the polynomial
    \begin{equation}\label{b2_aux_0}
        P_{a, b, c}(x) = x^{2a + 2b + c} + x^{2a + c} + x^{2b + c} + x^c - x^{a + b + 2c} - x^{a + b} - 2x^{a + b + c}
    \end{equation}
    is not divisible by $\Phi_f(x)$ for any $f \ge 2$ such that $f \mid m$.
\end{lemma}
\begin{proof}
    By Lemma~\ref{simple_nut} and Corollary~\ref{zeroeigenvalue}, we have that $B_2(m; a, b, c)$ is a nut graph if and only if
    \begin{equation}\label{b2_aux_1}
        (\zeta^a + \zeta^{-a})(\zeta^b + \zeta^{-b}) = |1 + \zeta^c|^2
    \end{equation}
    holds for exactly one $m$-th root of unity $\zeta$. Since $|1 + \zeta^c|^2 = (1 + \zeta^c)(1 + \zeta^{-c})$ and \eqref{b2_aux_1} holds for $\zeta = 1$, we conclude that $B_2(m; a, b, c)$ is a nut graph if and only if
    \begin{equation}\label{b2_aux_2}
        \zeta^{a + b} + \zeta^{a - b} + \zeta^{- a + b} + \zeta^{-a - b} - \zeta^c - \zeta^{-c} - 2 = 0
    \end{equation}
    is not true for any $m$-th root of unity $\zeta \neq 1$. The result now follows from the irreducibility of cyclotomic polynomials after multiplying~\eqref{b2_aux_2} by $\zeta^{a + b + c}$.
\end{proof}

Throughout the rest of the section, let $P_{a, b, c}(x)$ be the polynomial~\eqref{b2_aux_0} from Lemma~\ref{b2_poly_lemma}. We proceed by showing that conditions (i) and (ii) from Theorem~\ref{b2_main_th} are necessary.

\begin{lemma}\label{b2_cond_lemma}
    If at least one of the conditions (i) and (ii) from Theorem \ref{b2_main_th} does not hold, then the graph $B_2(m; a, b, c)$ is not a nut graph.
\end{lemma}
\begin{proof}
    First, suppose that condition (i) does not hold. If $m$ is not coprime to $\gcd(a - b, a + b + c)$, then $m$, $a - b$ and $a + b + c$ all share a common prime factor $p$. By letting $\zeta$ be any primitive $p$-th root of unity, we get
    \[
        P_{a, b, c}(\zeta) = (\zeta^{2a + 2b + c} - \zeta^{a + b}) + (\zeta^{2a + c} + \zeta^{2b + c} - 2\zeta^{a + b + c}) + (\zeta^c - \zeta^{a + b + 2c}) = 0. 
    \]
    Since $p \mid m$, Lemma~\ref{b2_poly_lemma} implies that $B_2(m; a, b, c)$ is not a nut graph. The same conclusion can be analogously obtained if $m$ is not coprime to $\gcd(a - b, a + b - c)$, $\gcd(a + b, a - b + c)$ or $\gcd(a + b, a - b - c)$.

    Now, suppose that condition (ii) does not hold. Without loss of generality, let $v_2(a) = \ell$, $v_2(c) = \ell + 1$, $v_2(m) \ge \ell + 2$ and let $\zeta$ be a primitive $2^{\ell + 2}$-th root of unity. Thus, we have $\zeta^{2a} = -1$ and $\zeta^c = -1$, so that
    \[
        P_{a, b, c}(\zeta) = \zeta^c (\zeta^{2a} + 1)(\zeta^{2b} + 1) - \zeta^{a + b} (\zeta^c + 1)^2 = 0.
    \]
    Since $2^{\ell + 2} \mid m$, Lemma~\ref{b2_poly_lemma} implies that $B_2(m; a, b, c)$ is not a nut graph.
\end{proof}

Our strategy is to suppose that conditions (i) and (ii) from Theorem~\ref{b2_main_th} hold and then prove that $B_2(m; a, b, c)$ is a nut graph if and only if conditions (iii) and (iv) also hold. To do this, we will need the following three lemmas on the divisibility of $P_{a, b, c}(x)$ by cyclotomic polynomials.

\begin{lemma}\label{b2_square_lemma}
    Suppose that conditions (i) and (ii) from Theorem \ref{b2_main_th} hold and let $f \ge 2$ be such that $f \mid m$ and $p^2 \mid f$ for some odd prime $p$. Then $\Phi_f(x) \nmid P_{a, b, c}(x)$.
\end{lemma}
\begin{proof}
    By way of contradiction, suppose that $\Phi_f(x) \mid P_{a, b, c}(x)$. Since $\Phi_f(x) = \Phi_{f / p}(x^p)$, Lemma~\ref{cool_div_lemma} yields a contradiction provided one of the numbers $2a + 2b + c, 2a + c, 2b + c, c, a + b + 2c, \linebreak a + b, a + b + c$ has a unique remainder modulo $p$. With this in mind, we finalize the proof by splitting the problem into four cases.

    \bigskip\noindent
    \textbf{Case 1:} $p \mid a - b$ and $p \mid a + b$.\quad
    Here, we have $p \mid a, b$, hence $p \mid c$ cannot be true, since otherwise condition (i) would not hold. Therefore, $a + b$ has a unique remainder modulo $p$.

    \bigskip\noindent
    \textbf{Case 2:} $p \mid a - b$ and $p \nmid a + b$.\quad
    In this case, we have $p \nmid a, b$, which means that $c$ has a different remainder modulo $p$ from $2a + 2b + c$, $2a + c$, $2b + c$ and $a + b + c$. Furthermore, $c$ also has a different remainder modulo $p$ from $a + b + 2c$ and $a + b$, since otherwise condition (i) would not hold. Thus, $c$ has a unique remainder modulo $p$.

    \bigskip\noindent
    \textbf{Case 3:} $p \nmid a - b$ and $p \mid a + b$. Here, we have $p \nmid a, b$, and it is not difficult to verify that $2a + c$ has a unique remainder modulo $p$.

    \bigskip\noindent
    \textbf{Case 4:} $p \nmid a - b$ and $p \nmid a + b$. If $p \nmid c$, it trivially follows that $a + b + c$ has a unique remainder modulo $p$. Now, suppose that $p \mid c$. In this case, $a + b + 2c$, $a + b$ and $a + b + c$ form an equivalence class with respect to congruence modulo $p$, hence Lemma \ref{cool_div_lemma} yields
    \[
        \Phi_f(x) \mid -x^{a + b + 2c} - x^{a + b} - 2x^{a + b + c} \quad \mbox{and} \quad \Phi_f(x) \mid x^{2a + 2b + c} + x^{2a + c} + x^{2b + c} + x^c.
    \]
    Therefore, we have $\Phi_f(x) \mid (x^c + 1)^2$ and $\Phi_f(x) \mid (x^{2a} + 1)(x^{2b} + 1)$. We obtain $\Phi_f(x) \mid x^c + 1$ from the irreducibility of $\Phi_f(x)$, and we may also assume without loss of generality that $\Phi_f(x) \mid x^{2a} + 1$. From here, we get $v_2(c) = v_2(f) - 1$ and $v_2(2a) = v_2(f) - 1$, which means that condition~(ii) does not hold, thus yielding a contradiction.
\end{proof}

\begin{lemma}\label{b2_two_lemma}
    Suppose that conditions (i) and (ii) from Theorem \ref{b2_main_th} hold and let $f \ge 2$ be such that $f \mid m$ and $16 \mid f$. Then $\Phi_f(x) \nmid P_{a, b, c}(x)$.
\end{lemma}
\begin{proof}
     By way of contradiction, suppose that $\Phi_f(x) \mid P_{a, b, c}(x)$. Since $\Phi_f(x) = \Phi_{f / 8}(x^8)$, Lemma~\ref{cool_div_lemma} implies that to finalize the proof, it is sufficient to show that one of the numbers $2a + 2b + c, 2a + c, 2b + c, c, a + b + 2c, a + b, a + b + c$ has a unique remainder modulo $8$. We proceed by splitting the problem into four cases.

    \bigskip\noindent
    \textbf{Case 1:} $8 \mid a - b$ and $8 \mid a + b$.\quad
    In this case, we have $4 \mid a, b$, hence $4 \nmid c$ cannot be true, since otherwise condition (i) would not hold. Thus, $a + b$ has a unique remainder modulo $8$.

    \bigskip\noindent
    \textbf{Case 2:} $8 \mid a - b$ and $8 \nmid a + b$.\quad
    Here, we have $4 \nmid a, b$, which means that $c$ has a different remainder modulo $8$ from $2a + c$ and $2b + c$. If $4 \nmid a + b$, then it can be proved analogously to Case~2 of Lemma~\ref{b2_square_lemma} that $c$ has a unique remainder modulo $8$. Now, suppose that $4 \mid a + b$. In this scenario, $a$ and $b$ are both even, hence $c$ must be odd in accordance with condition~(i). Therefore, $a + b + 2c$ and $a + b$ are the only two even numbers and both of them have a unique remainder modulo $8$.

    \bigskip\noindent
    \textbf{Case 3:} $8 \nmid a - b$ and $8 \mid a + b$.\quad
    If $4 \nmid a - b$, it is straightforward to observe that $2a + c$ has a unique remainder modulo $8$. On the other hand, if $4 \mid a - b$, then $a$ and $b$ are both even, hence $c$ must be odd. It follows that $a + b + 2c$ and $a + b$ are the only two even numbers, with both of them having a unique remainder modulo $8$.

    \bigskip\noindent
    \textbf{Case 4:} $8 \nmid a - b$ and $8 \nmid a + b$.\quad
    This case can be resolved analogously to Case~4 from Lemma~\ref{b2_square_lemma}.
\end{proof}

\begin{lemma}\label{b2_prime_lemma}
    Suppose that conditions (i) and (ii) from Theorem \ref{b2_main_th} hold and let $p \ge 11$ be a prime such that $p \mid m$. Then $\Phi_p(x) \nmid P_{a, b, c}(x)$.
\end{lemma}
\begin{proof}
    By way of contradiction, suppose that $\Phi_p(x) \mid P_{a, b, c}(x)$. It trivially follows that the auxiliary polynomial
    \begin{align*}
        P_{a, b, c}^{\bmod p}(x) = x^{(2a + 2b + c) \bmod p} &+ x^{(2a + c) \bmod p} + x^{(2b + c) \bmod p} + x^{c \bmod p}\\
        &- x^{(a + b + 2c) \bmod p} - x^{(a + b) \bmod p} - 2x^{(a + b + c) \bmod p}
    \end{align*}
    is then also divisible by $\Phi_p(x)$. Observe that $\Phi_p(x) = \sum_{j = 0}^{p - 1} x^j$. Since $\deg P_{a, b, c}^{\bmod p} \le p - 1 = \deg \Phi_p$, we have that either $P_{a, b, c}^{\bmod p}(x) \equiv 0$ or $P_{a, b, c}^{\bmod p}(x) = \beta \, \Phi_p(x)$ for some $\beta \in \mathbb{Q} \setminus \{ 0 \}$. By following the case analysis from Lemma~\ref{b2_square_lemma}, it is not difficult to establish that one of the numbers $2a + 2b + c, 2a + c, 2b + c, c, a + b + 2c, a + b, a + b + c$ must have a unique remainder modulo $p$. Thus, $P_{a, b, c}^{\bmod p}(x) \equiv 0$ cannot hold. On the other hand, if $P_{a, b, c}^{\bmod p}(x) = \beta \, \Phi_p(x)$ for some $\beta \in \mathbb{Q} \setminus \{ 0 \}$, then $P_{a, b, c}^{\bmod p}(x)$ has $p$ nonzero terms, which is impossible since $p \ge 11$.
\end{proof}

We now use the previously derived Lemmas~\ref{b2_poly_lemma} and \ref{b2_square_lemma}--\ref{b2_prime_lemma} together with Theorem \ref{filaschinz_th} to obtain the following result.

\begin{lemma}\label{b2_key_lemma}
    Suppose that conditions (i) and (ii) from Theorem \ref{b2_main_th} hold. Then the graph $B_2(m; a, b, c)$ is a nut graph if and only if $\Phi_f(x) \nmid P_{a, b, c}(x)$ for every
    \begin{equation}\label{b2_aux_3}
        f \in \{2, 3, 4, 5, 6, 7, 8, 10, 12, 14, 15, 20, 24, 28, 30, 40, 56, 60, 120\}
    \end{equation}
    such that $f \mid m$.
\end{lemma}
\begin{proof}
    Let $F$ be the set from \eqref{b2_aux_3}. If $\Phi_f(x) \mid P_{a, b, c}(x)$ is satisfied for some $f \in F$ such that $f \mid m$, then Lemma~\ref{b2_poly_lemma} implies that $B_2(m; a, b, c)$ is not a nut graph. Thus, it suffices to prove the converse.
    
    Suppose that $B_2(m; a, b, c)$ is not a nut graph. By Lemma~\ref{b2_poly_lemma}, there exists an $f \ge 2$ such that $f \mid m$ and $\Phi_f(x) \mid P_{a, b, c}(x)$. Moreover, Lemmas \ref{b2_square_lemma} and \ref{b2_two_lemma} imply that $16 \nmid f$ and $p^2 \nmid f$ for every odd prime $p$. We now divide the problem into two cases depending on whether $f$ has a prime factor below $11$.

    \bigskip\noindent
    \textbf{Case 1:} $\gcd(f, 2 \cdot 3 \cdot 5 \cdot 7) = 1$.\quad
    Here, Theorem~\ref{filaschinz_th} can be repeatedly used to cancel out prime factors of $f$ until only one factor is left. In other words, we get $\Phi_p \mid P_{a, b, c}(x)$ for some prime $p \ge 11$ such that $p \mid f$. This leads to a contradiction due to Lemma \ref{b2_prime_lemma}.

    \bigskip\noindent
    \textbf{Case 2:} $\gcd(f, 2 \cdot 3 \cdot 5 \cdot 7) > 1$.\quad
    In this case, Theorem~\ref{filaschinz_th} can be repeatedly applied to cancel out any potential prime factors of $f$ that are above $7$. Thus, we get $\Phi_{f'}(x) \mid P_{a, b, c}(x)$ for some $f' \ge 2$ without a prime factor above $7$ and such that $f' \mid f$. Additionally, since $(5 - 2) + (7 - 2) > 7 - 2$, we can cancel out either $5$ or $7$ as a prime factor if they appear as simultaneous factors. The same can be done to the potentially simultaneous factors $3$ and $7$ because $(3 - 2) + (7 - 2) > 7 - 2$. With this in mind, we conclude that $\Phi_{f''}(x) \mid P_{a, b, c}(x)$ holds for some $f'' \ge 2$ without a prime factor above $7$ or repeating odd prime factors and such that $f'' \mid m$, $v_2(f'') \le 3$, $21 \nmid f''$ and $35 \nmid f''$. The result follows by noting that $f'' \in F$.
\end{proof}

We are finally in a position to complete the proof of Theorem \ref{b2_main_th} through Lemmas \ref{b2_cond_lemma} and~\ref{b2_key_lemma}.

\begin{proof}[Proof of Theorem \ref{b2_main_th}]
    If at least one of the conditions (i) and (ii) from Theorem \ref{b2_main_th} does not hold, then $B_2(m; a, b, c)$ is indeed not a nut graph, due to Lemma \ref{b2_cond_lemma}. Now, suppose that both of these conditions hold. For any $f \in \mathbb{N}$, let
    \begin{align*}
        P_{a, b, c}^{\bmod f}(x) = x^{(2a + 2b + c) \bmod f} &+ x^{(2a + c) \bmod f} + x^{(2b + c) \bmod f} + x^{c \bmod f}\\
        &- x^{(a + b + 2c) \bmod f} - x^{(a + b) \bmod f} - 2x^{(a + b + c) \bmod f} .
    \end{align*}
    We trivially observe that $\Phi_f(x) \mid P_{a, b, c}(x)$ if and only if $\Phi_f(x) \mid P_{a, b, c}^{\bmod f}(x)$. By Lemma~\ref{b2_key_lemma}, we conclude that $B_2(m; a, b, c)$ is a nut graph if and only if certain modular conditions on $a + b$, $a - b$ and $c$ are satisfied. Since the set from \eqref{b2_aux_3} is finite, a computer can now be used to find all the triples $(a \bmod f, b \bmod f, c \bmod f)$ that yield a $P_{a, b, c}^{\bmod f}(x)$ divisible by $\Phi_f(x)$.
    
    The execution results from the \texttt{SageMath} \cite{SageMath} script given in Appendix \ref{comp_b2} imply that, provided conditions (i) and (ii) hold, $\Phi_f(x) \mid P_{a, b, c}^{\bmod f}(x)$ cannot be true when $f \mid m$, except for $f = 12$ and $f = 30$. For $f = 12$, the obtained results imply that $\Phi_{12}(x) \nmid P_{a, b, c}^{\bmod 12}(x)$ holds if and only if $(a + b, a - b, c) \not\equiv_{12} (\pm 2, \pm 2, \pm 3)$, while for $f = 30$, we have that $\Phi_{30}(x) \nmid P_{a, b, c}^{\bmod 30}(x)$ holds if and only if $(a + b, a - b, c) \not\equiv_{30} (\pm 3, \pm 5, \pm 6), (\pm 5, \pm 3, \pm 6), (\pm 3, \pm 9, \pm 10), (\pm 9, \pm 3, \pm 10), \linebreak (\pm 5, \pm 9, \pm 12), (\pm 9, \pm 5, \pm 12)$. These observations coincide with conditions (iii) and (iv) from Theorem \ref{b2_main_th}, respectively.
\end{proof}

We end the section with the following two examples.

\begin{example}
    The graph $B_2(24; 4, 6, 3)$, where $m = 24$, $a = 4$, $b = 6$ and $c = 3$, is not a nut graph since we have
    \[
        a + b = 10 \equiv_{12} -2, \qquad a - b = -2 \equiv_{12} -2 \qquad \mbox{and} \qquad c = 3 \equiv_{12} 3,
    \]
    which contradicts condition (iii) from Theorem \ref{b2_main_th}. This fact can also be manually checked via Lemma \ref{b2_poly_lemma} by observing that
    \begin{align*}
        P_{4, 6, 3}(x) &= x^{23} + x^{11} + x^{15} + x^3 - x^{16} - x^{10} - 2x^{13}\\
        &= (x^4 - x^2 + 1) (x^{19} + x^{17} - x^{13} - x^{12} - x^{10} - x^9 + x^5 + x^3),
    \end{align*}
    where $\Phi_{12}(x) = x^4 - x^2 + 1$.
\end{example}
\begin{example}
    The graph $B_2(30; 1, 4, 6)$, where $m = 30$, $a = 1$, $b = 4$ and $c = 6$, is not a nut graph since we have
    \[
        a + b = 5 \equiv_{30} 5, \qquad a - b = -3 \equiv_{30} -3 \qquad \mbox{and} \qquad c = 6 \equiv_{30} 6,
    \]
    which contradicts condition (iv) from Theorem \ref{b2_main_th}. This fact can also be manually checked via Lemma \ref{b2_poly_lemma} by observing that
    \begin{align*}
        P_{1, 4, 6}(x) &= x^{16} + x^{8} + x^{14} + x^6 - x^{17} - x^{5} - 2x^{11}\\
        &= (x^8 + x^7 - x^5 - x^4 - x^3 + x + 1) (-x^9 + 2x^8 - 2x^7 + 2x^6 - x^5),
    \end{align*}
    where $\Phi_{30}(x) = x^8 + x^7 - x^5 - x^4 - x^3 + x + 1$.
\end{example}

\section{The class \texorpdfstring{$\cB_3$}{B3}}

In this section we consider the class $\cB_3$. Recall that these graphs are of the form $\bicirc(m; S, T, \linebreak R)$ with even $n \ge 4$, $S = T = \{ \frac{m}{2} \}$ and $R = \{0, a, b\}$, where $1 \le a < b < n$.
We may also assume that $a$ and $b$ are of the same parity.
For such parameters,
we denoted the graph $\bicirc(m; S, T, R)$ by $B_3(m; a, b)$.
The following theorem provides a classification of nut graphs among members of $\cB_3$.

\begin{theorem}\label{b3_main_th}
Let $m \ge 4$ be even and let $a$ and $b$ be integers of the same parity such that $1 \le a < b < m$. Then the graph $B_3(m; a, b)$ is a nut graph if and only if $a$ and $b$ are odd, while $\gcd(m, a) = \gcd(m, b) = 1$ and $v_2(b - a) \ge v_2(m)$.
\end{theorem}

The purpose of the present section is to give the proof of Theorem \ref{b3_main_th}. We start by observing the next result.

\begin{lemma}\label{b3_poly_lemma}
    The graph $B_3(m; a, b)$ is a nut graph if and only if the numbers $a$ and $b$ are odd, while the polynomial
    \begin{equation}\label{b3_aux_0}
        (x^{b - a} + 1)(x^a + 1)(x^b + 1)
    \end{equation}
    is not divisible by $\Phi_f(x)$ for any $f \ge 3$ such that $f \mid m$.
\end{lemma}
\begin{proof}
    By Lemma~\ref{simple_nut} and Corollary~\ref{zeroeigenvalue}, we have that $B_3(m; a, b)$ is a nut graph if and only if
    \begin{equation}\label{b3_aux_1}
        \zeta^\frac{m}{2} \cdot \zeta^\frac{m}{2} = |1 + \zeta^a + \zeta^b|^2
    \end{equation}
    holds for exactly one $m$-th root of unity $\zeta$. Note that \eqref{b3_aux_1} is not satisfied for $\zeta = 1$, and if \eqref{b3_aux_1} holds for a nonreal $\zeta$, then it also holds for $\overline{\zeta} \neq \zeta$. Therefore, since $|1 + \zeta^a + \zeta^b|^2 = (1 + \zeta^a + \zeta^b)(1 + \zeta^{-a} + \zeta^{-b})$, we conclude that $B_3(m; a, b)$ is a nut graph if and only if
    \begin{equation}\label{b3_aux_2}
        \zeta^a + \zeta^{-a} + \zeta^b + \zeta^{-b} + \zeta^{a - b} + \zeta^{-a + b} + 2 = 0
    \end{equation}
    holds only for $\zeta = -1$ among all the $m$-th roots of unity $\zeta$. Clearly, \eqref{b3_aux_2} is satisfied for $\zeta = -1$ if and only if $a$ and $b$ are odd. The result now follows from the irreducibility of cyclotomic polynomials after multiplying \eqref{b3_aux_2} by $\zeta^b$ and factoring accordingly.
\end{proof}

We now in a position to complete the proof of Theorem \ref{b3_main_th}.

\begin{proof}[Proof of Theorem \ref{b3_main_th}]
If $a$ and $b$ are even, then Lemma~\ref{b3_poly_lemma} implies that $B_3(m; a, b)$ is not a nut graph, which agrees with Theorem \ref{b3_main_th}. Now, suppose that the numbers $a$ and $b$ are odd. In this case, we trivially observe through Lemma~\ref{b3_poly_lemma} that $B_3(m; a, b)$ is a nut graph if and only if each of the equations $x^a = -1$, $x^b = -1$ and $x^{b - a} = -1$ has at most one solution among the $m$-th roots of unity, namely $-1$.

For any $\ell \in \mathbb{N}$, observe that the equation $x^\ell = -1$ in $x \in \mathbb{C}, \, x^m = 1$, is equivalent to the equation
\begin{equation}\label{aux_13}
   k \ell \equiv_m \frac{m}{2}
\end{equation}
in $k \in \Z_m$.
Thus, these two equations have the same number of solutions. Equation \eqref{aux_13} is a linear congruence relation, which is solvable if and only if $\gcd(m, \ell) \mid \frac{m}{2}$ (see, for example, \cite[p.~170, Theorem~5.14]{Tattersall}).
Moreover, in the case it is solvable, it has $\gcd(m, \ell)$ distinct solutions in $\Z_m$. 
Thus, the solution to $x^\ell = -1$ exists if and only if $\gcd(m, \ell) \mid \frac{m}{2}$ and it is unique if and only if $\gcd(m, \ell) = 1$.

Since $a$ and $b$ are both odd, $B_3(m; a, b)$ is a nut graph if and only if the equations $x^a = -1$ and $x^b = -1$ have the unique solution $-1$ among the $m$-th roots of unity, while $x^{b - a} = -1$ has no solutions over the same set. In other words, $B_3(m; a, b)$ is a nut graph if and only if $\gcd(m, a) = 1$, $\gcd(m, b) = 1$ and $\gcd(m, b - a) \nmid \frac{m}{2}$. Note that $\gcd(m, b - a) \nmid \frac{m}{2}$ is satisfied if and only if $v_2(b - a) \ge v_2(m)$.
\end{proof}

\section{Conclusion}

In this paper, we have given the complete classification of quartic bicirculant nut graphs (abbrev.\ QBN graphs). Using a computer, we have also enumerated the connected quartic bicirculants, the nonbipartite connected quartic bicirculants, and the QBN graphs up to order $50$. Among these nut graphs we have identified the vertex-transitive (abbrev.~VT), Cayley and circulant graphs. These results are also stratified with respect to the classes $\cB_1$, $\cB_2$ and $\cB_3$; see Table \ref{tab:enumeration}. From the table, we can see that the unique smallest circulant QBN graph is $B_2(4; 1, 1, 1) \cong \Circ(8, \{1, 2\})$ in Figure~\ref{fig:smallQBN}(a). The unique smallest non-VT QBN graph is $B_2(6; 1, 2, 3)$; see Figure~\ref{fig:smallQBN}(b). The smallest non-Cayley VT QBN graph is also unique; it is the graph $B_2(15; 3, 6, 5)$ in Figure~\ref{fig:smallQBN}(c). This is the ``standard'' way of drawing
a bicirculant; an alternative drawing of the same graph is in Figure~\ref{fig:smallQBN}(d). This graph has also been reported in \cite[Figure~5(b)]{BaFowPi2024}. Up to order $50$, it is the only example of a non-Cayley VT QBN graph.

There are three nonisomorphic
noncirculant Cayley QBN graphs of order 20; these are the graphs $B_2(10; 1, 3, 5)$, $B_2(10; 2, 4, 5)$ and $B_3(10; 1, 3)$;
see Figure~\ref{fig:smallQBN}(e)--(g). We have $B_3(10; 1, 3) \cong \Cay(\Dih(10), \{r^5, s, sr, sr^3\})$, where $\Dih(10) = \langle r, s \mid r^{10} = s^2 = (sr)^2 = 1 \rangle$ (i.e., $\Dih(10)$ is the dihedral group of order $20$). On the other hand, $B_2(10, 1, 3, 5) \cong \Cay(F_{20}, \{rs, s^2r^{-1}, s^2\})$ and $B_2(10; 2, 4, 5) \cong \Cay(F_{20}, \{r, s\})$, where $F_{20} = \langle r, s \mid r^{5} = s^4 = srs^{-1}r^{-2} = 1 \rangle \cong \mathbb{Z}_5 \rtimes \mathbb{Z}_4$ (i.e., $F_{20}$ is the unique Frobenius group of order $20$).

\begin{table}[!ht]
\centering
\subcaptionbox{All classes}{
\begin{tabular}{r|r|r|r||r|r|r}
$n$ & $\mathfrak{C}_n$ & $\mathfrak{B}_n$ & $\mathfrak{N}_n$ & $\mathfrak{V}_n$ & $\mathfrak{Y}_n$ & $\mathfrak{Z}_n$ \\
\hline \hline
8 & 3 & 2 & 1 & 1 & 1 & 1\\
10 & 3 & 2 & 1 & 1 & 1 & 1\\
12 & 12 & 9 & 3 & 2 & 2 & 2\\
14 & 8 & 6 & 5 & 2 & 2 & 2\\
16 & 17 & 12 & 6 & 3 & 3 & 3\\
18 & 17 & 13 & 7 & 2 & 2 & 2\\
20 & 35 & 25 & 17 & 7 & 7 & 4\\
22 & 19 & 15 & 14 & 4 & 4 & 4\\
24 & 69 & 48 & 18 & 6 & 6 & 4\\
26 & 28 & 21 & 20 & 5 & 5 & 5\\
28 & 64 & 44 & 36 & 8 & 8 & 6\\
30 & 69 & 52 & 27 & 6 & 5 & 4\\
32 & 71 & 48 & 30 & 7 & 7 & 7\\
34 & 47 & 36 & 35 & 7 & 7 & 7\\
36 & 133 & 91 & 53 & 8 & 8 & 6\\
38 & 59 & 45 & 44 & 8 & 8 & 8\\
40 & 159 & 107 & 63 & 13 & 13 & 8\\
42 & 125 & 94 & 55 & 8 & 8 & 6\\
44 & 151 & 103 & 93 & 14 & 14 & 10\\
46 & 86 & 66 & 65 & 10 & 10 & 10\\
48 & 266 & 173 & 76 & 12 & 12 & 8\\
50 & 122 & 93 & 78 & 9 & 9 & 9\\
\end{tabular}
}
\hspace{1cm}
\subcaptionbox{Class $\cB_2$}{
\begin{tabular}{r|r|r|r||r|r|r}
$n$ & $\mathfrak{C}_n$ & $\mathfrak{B}_n$ & $\mathfrak{N}_n$ & $\mathfrak{V}_n$ & $\mathfrak{Y}_n$ & $\mathfrak{Z}_n$ \\
\hline \hline
8 & 2 & 1 & 1 & 1 & 1 & 1\\
10 & 2 & 2 & 1 & 1 & 1 & 1\\
12 & 8 & 7 & 3 & 2 & 2 & 2\\
14 & 6 & 6 & 5 & 2 & 2 & 2\\
16 & 11 & 8 & 6 & 3 & 3 & 3\\
18 & 13 & 13 & 7 & 2 & 2 & 2\\
20 & 23 & 20 & 16 & 6 & 6 & 4\\
22 & 15 & 15 & 14 & 4 & 4 & 4\\
24 & 46 & 37 & 17 & 5 & 5 & 4\\
26 & 21 & 21 & 20 & 5 & 5 & 5\\
28 & 42 & 37 & 33 & 6 & 6 & 6\\
30 & 52 & 52 & 27 & 6 & 5 & 4\\
32 & 49 & 39 & 30 & 7 & 7 & 7\\
34 & 36 & 36 & 35 & 7 & 7 & 7\\
36 & 90 & 79 & 49 & 6 & 6 & 6\\
38 & 45 & 45 & 44 & 8 & 8 & 8\\
40 & 111 & 91 & 61 & 11 & 11 & 8\\
42 & 94 & 94 & 55 & 8 & 8 & 6\\
44 & 104 & 91 & 87 & 10 & 10 & 10\\
46 & 66 & 66 & 65 & 10 & 10 & 10\\
48 & 185 & 150 & 75 & 11 & 11 & 8\\
50 & 93 & 93 & 78 & 9 & 9 & 9\\
\end{tabular}
}

\bigskip

\subcaptionbox{Class $\cB_1$}{
\begin{tabular}{r|r|r|r||r|r|r}
$n$ & $\mathfrak{C}_n$ & $\mathfrak{B}_n$ & $\mathfrak{N}_n$ & $\mathfrak{V}_n$ & $\mathfrak{Y}_n$ & $\mathfrak{Z}_n$ \\
\hline \hline
8 & 1 & 1 & 0 & 0 & 0 & 0\\
12 & 3 & 2 & 1 & 1 & 1 & 1\\
16 & 3 & 3 & 0 & 0 & 0 & 0\\
20 & 6 & 4 & 1 & 1 & 1 & 1\\
24 & 7 & 7 & 0 & 0 & 0 & 0\\
28 & 7 & 5 & 2 & 1 & 1 & 1\\
32 & 6 & 6 & 0 & 0 & 0 & 0\\
36 & 11 & 8 & 3 & 1 & 1 & 1\\
40 & 10 & 10 & 0 & 0 & 0 & 0\\
44 & 11 & 8 & 3 & 1 & 1 & 1\\
48 & 14 & 14 & 0 & 0 & 0 & 0\\
\end{tabular}
}
\hspace{1cm}
\subcaptionbox{Class $\cB_3$}{
\begin{tabular}{r|r|r|r||r|r|r}
$n$ & $\mathfrak{C}_n$ & $\mathfrak{B}_n$ & $\mathfrak{N}_n$ & $\mathfrak{V}_n$ & $\mathfrak{Y}_n$ & $\mathfrak{Z}_n$ \\
\hline \hline
8 & 1 & 1 & 0 & 0 & 0 & 0\\
12 & 3 & 2 & 1 & 1 & 1 & 1\\
16 & 3 & 3 & 0 & 0 & 0 & 0\\
20 & 4 & 3 & 2 & 2 & 2 & 1\\
24 & 5 & 5 & 1 & 1 & 1 & 0\\
28 & 6 & 4 & 3 & 3 & 3 & 1\\
32 & 5 & 5 & 0 & 0 & 0 & 0\\
36 & 8 & 6 & 3 & 3 & 3 & 1\\
40 & 7 & 7 & 2 & 2 & 2 & 0\\
44 & 8 & 6 & 5 & 5 & 5 & 1\\
48 & 11 & 11 & 1 & 1 & 1 & 0\\
\end{tabular}
}
\caption{$\mathfrak{C}_n$, $\mathfrak{B}_n$, $\mathfrak{N}_n$, $\mathfrak{V}_n$, $\mathfrak{Y}_n$ and $\mathfrak{Z}_n$ denote the number (up to isomorphism) of connected quartic bicirculants, nonbipartite connected quartic bicirculants, QBN graphs, VT QBN graphs, Cayley QBN graphs and circulant QBN graphs, of order $n$, respectively.}
\label{tab:enumeration}
\end{table}

\begin{figure}[!htbp]
\centering
\subcaptionbox{$\Circ(8, \{1, 2\})$}{
\quad\begin{tikzpicture}[scale=1.8]
\tikzstyle{vertex}=[draw,circle,font=\scriptsize,minimum size=3pt,inner sep=1pt,fill=magenta!80!white]
\tikzstyle{edge}=[draw,thick]
\foreach \i in {1,3,...,7} {
	\node[vertex] (v\i) at ({45 *\i + 22.5}:1) {};
}
\foreach \i in {0,2,...,6} {
	\node[vertex,fill=green!80!white] (v\i) at ({45 *\i + 22.5}:1) {};
}
\foreach \i in {0,1,...,7} {
\pgfmathtruncatemacro{\j}{mod(\i + 1, 8)}
\pgfmathtruncatemacro{\k}{mod(\i + 2, 8)}
\path[edge] (v\i) -- (v\j);
\path[edge] (v\i) -- (v\k);
}
\end{tikzpicture}\quad
}
\quad
\subcaptionbox{$B_2(6; 1, 2, 3)$}{
\begin{tikzpicture}[scale=1.0]
\tikzstyle{vertex}=[draw,circle,font=\scriptsize,minimum size=3pt,inner sep=1pt,fill=magenta!80!white]
\tikzstyle{edge}=[draw,thick]
\foreach \i in {0,1,...,5} {
	\node[vertex] (v\i) at ({60 *\i + 0}:2) {};
	\node[vertex,fill=green!80!white] (u\i) at ({60 *\i + 20}:0.7) {};
}
\foreach \i in {0,1,...,5} {
\pgfmathtruncatemacro{\j}{mod(\i + 2, 6)}
\pgfmathtruncatemacro{\k}{mod(\i + 1, 6)}
\pgfmathtruncatemacro{\l}{mod(\i + 3, 6)}
\path[edge] (v\i) -- (v\j);
\path[edge] (u\i) -- (u\k);
\path[edge] (v\i) -- (u\i);
\path[edge] (v\i) -- (u\l);
}
\end{tikzpicture}
}
\quad
\subcaptionbox{$B_2(15; 3, 6, 5)$}{
\begin{tikzpicture}[scale=0.85,rotate=90]
\tikzstyle{vertex}=[draw,circle,font=\scriptsize,minimum size=3pt,inner sep=1pt,fill=magenta!80!white]
\tikzstyle{edge}=[draw,thick]
\foreach \i in {0,1,...,14} {
	\node[vertex] (v\i) at ({24*\i + 0}:2.2) {};
	\node[vertex,fill=green!80!white] (u\i) at ({24*\i - 2 * 24 - 12}:0.8) {};
}
\foreach \i in {0,1,...,14} {
\pgfmathtruncatemacro{\j}{mod(\i + 3, 15)}
\pgfmathtruncatemacro{\k}{mod(\i + 6, 15)}
\pgfmathtruncatemacro{\l}{mod(\i + 5, 15)}
\path[edge] (v\i) -- (v\j);
\path[edge] (u\i) -- (u\k);
\path[edge] (v\i) -- (u\i);
\path[edge] (v\i) -- (u\l);
}
\end{tikzpicture}
}

\bigskip
\subcaptionbox{$B_2(15; 3, 6, 5)$}{
\begin{tikzpicture}[scale=0.8]
\pgfmathtruncatemacro{\off}{7}
\tikzstyle{vertex}=[draw,circle,font=\scriptsize,minimum size=3pt,inner sep=1pt]
\tikzstyle{edge}=[draw,thick]
\tikzstyle{col1}=[fill=magenta!80!white]
\tikzstyle{col2}=[fill=green!80!white]
\foreach[count=\i] \x/\color in {17/col2,26/col2,20/col2,29/col2,23/col2} {
	\node[vertex,\color] (v\x) at ($ (0*60:2) + ({72*\i + \off}:0.6) $) {};
}
\foreach[count=\i] \x/\color in {2/col1,11/col1,5/col1,14/col1,8/col1} {
	\node[vertex,\color] (v\x) at ($ (1*60:2) + ({72*\i + \off}:0.6) $) {};
}
\foreach[count=\i] \x/\color in {22/col2,16/col2,25/col2,19/col2,28/col2} {
	\node[vertex,\color] (v\x) at ($ (2*60:2) + ({72*\i + \off}:0.6) $) {};
}
\foreach[count=\i] \x/\color in {7/col1,1/col1,10/col1,4/col1,13/col1} {
	\node[vertex,\color] (v\x) at ($ (3*60:2) + ({72*\i + \off}:0.6) + (-0.5, 0) $) {};
}
\foreach[count=\i] \x/\color in {27/col2,21/col2,15/col2,24/col2, 18/col2} {
	\node[vertex,\color] (v\x) at ($ (4*60:2) + ({72*\i + \off}:0.6) + (-0.3, -0.3)$) {};
}
\foreach[count=\i] \x/\color in {12/col1,6/col1,0/col1,9/col1,3/col1} {
	\node[vertex,\color] (v\x) at ($ (5*60:2) + ({72*\i + \off}:0.6) + (0.3, -0.3)$) {};
}
\path[edge] (v2) -- (v17);
\path[edge] (v11) -- (v26);
\path[edge] (v5) -- (v20);
\path[edge] (v14) -- (v29);
\path[edge] (v8) -- (v23);
\path[edge] (v2) -- (v22);
\path[edge] (v11) -- (v16);
\path[edge] (v5) -- (v25);
\path[edge] (v14) -- (v19);
\path[edge] (v8) -- (v28);
\path[edge] (v7) -- (v22);
\path[edge] (v1) -- (v16);
\path[edge] (v10) -- (v25);
\path[edge] (v4) -- (v19);
\path[edge] (v13) -- (v28);
\path[edge] (v7) -- (v27);
\path[edge] (v1) -- (v21);
\path[edge] (v10) -- (v15);
\path[edge] (v4) -- (v24);
\path[edge] (v13) -- (v18);
\path[edge] (v12) -- (v27);
\path[edge] (v6) -- (v21);
\path[edge] (v0) -- (v15);
\path[edge] (v9) -- (v24);
\path[edge] (v3) -- (v18);
\path[edge] (v12) -- (v17);
\path[edge] (v6) -- (v26);
\path[edge] (v0) -- (v20);
\path[edge] (v9) -- (v29);
\path[edge] (v3) -- (v23);
\path[edge] (v2) -- (v5) -- (v8) -- (v11) -- (v14) -- (v2);
\path[edge] (v17) -- (v26) -- (v20) -- (v29) -- (v23) -- (v17);
\path[edge] (v12) -- (v0) -- (v3) -- (v6) -- (v9) -- (v12);
\path[edge] (v27) -- (v21) -- (v15) -- (v24) -- (v18) -- (v27);
\path[edge] (v7) -- (v10) -- (v13) -- (v1) -- (v4) -- (v7);
\path[edge] (v22) -- (v16) -- (v25) -- (v19) -- (v28) -- (v22);
\end{tikzpicture}
}\quad
\subcaptionbox{$B_3(10; 1, 3)$}{
\begin{tikzpicture}[scale=1.0]
\pgfmathtruncatemacro{\off}{9}
\tikzstyle{vertex}=[draw,circle,font=\scriptsize,minimum size=3pt,inner sep=1pt]
\tikzstyle{edge}=[draw,thick]
\tikzstyle{col1}=[fill=magenta!80!white]
\tikzstyle{col2}=[fill=green!80!white]
\foreach[count=\i] \x/\color in {0/col1,5/col2,13/col1,18/col2} {
	\node[vertex,\color] (v\x) at ($ (0*72:1.5) + ({90*\i + \off}:0.5) $) {};
}
\foreach[count=\i] \x/\color in {3/col2, 8/col1,11/col1,16/col2} {
	\node[vertex,\color] (v\x) at ($ (1*72:1.5) + ({90*\i + \off + 2*90}:0.5) $) {};
}
\foreach[count=\i] \x/\color in {1/col2,6/col1,19/col1,14/col2} {
	\node[vertex,\color] (v\x) at ($ (2*72:1.5) + ({90*\i + \off}:0.5) $) {};
}
\foreach[count=\i] \x/\color in { 9/col2,4/col1,12/col2,17/col1} {
	\node[vertex,\color] (v\x) at ($ (3*72:1.5) + ({-90*\i + \off + 90}:0.5) $) {};
}
\foreach[count=\i] \x/\color in {2/col1,7/col2,10/col2,15/col1} {
	\node[vertex,\color] (v\x) at ($ (4*72:1.5) + ({-90*\i + \off - 90}:0.5 ) $) {};
}
\path[edge] (v3) -- (v8) -- (v11) -- (v16) -- (v3);
\path[edge] (v1) -- (v6) -- (v19) -- (v14) -- (v1);
\path[edge] (v4) -- (v9) -- (v12) -- (v17) -- (v4);
\path[edge] (v10) -- (v7) -- (v2) -- (v15) -- (v10);
\path[edge] (v0) -- (v13) -- (v18) -- (v5) -- (v0);
\path[edge] (v5) -- (v16);
\path[edge] (v13) -- (v3);
\path[edge] (v18) -- (v8);
\path[edge] (v11) -- (v0);
\path[edge] (v6) -- (v16);
\path[edge] (v14) -- (v3);
\path[edge] (v19) -- (v8);
\path[edge] (v11) -- (v1);
\path[edge] (v17) -- (v6);
\path[edge] (v14) -- (v4);
\path[edge] (v19) -- (v9);
\path[edge] (v12) -- (v1);
\path[edge] (v17) -- (v7);
\path[edge] (v15) -- (v4);
\path[edge] (v10) -- (v9);
\path[edge] (v12) -- (v2);
\path[edge] (v2) -- (v13);
\path[edge] (v15) -- (v5);
\path[edge] (v7) -- (v18);
\path[edge] (v10) -- (v0);
\end{tikzpicture}
}
\subcaptionbox{$B_2(10; 1, 3, 5)$}{
\begin{tikzpicture}[scale=1.0]
\tikzstyle{vertex}=[draw,circle,font=\scriptsize,minimum size=3pt,inner sep=1pt,fill=magenta!80!white]
\tikzstyle{edge}=[draw,thick]
\foreach \i in {0,1,...,9} {
	\node[vertex] (v\i) at ({36 *\i + 0}:2) {};
	\node[vertex,fill=green!80!white] (u\i) at ({36 *\i + 165}:0.9) {};
}
\foreach \i in {0,1,...,9} {
\pgfmathtruncatemacro{\j}{mod(\i + 3, 10)}
\pgfmathtruncatemacro{\k}{mod(\i + 1, 10)}
\pgfmathtruncatemacro{\l}{mod(\i + 5, 10)}
\path[edge] (v\i) -- (v\j);
\path[edge] (u\i) -- (u\k);
\path[edge] (v\i) -- (u\i);
\path[edge] (v\i) -- (u\l);
}
\end{tikzpicture}
}

\bigskip
\subcaptionbox{$B_2(10; 2, 4, 5)$}{
\begin{tikzpicture}[scale=1.0]
\tikzstyle{vertex}=[draw,circle,font=\scriptsize,minimum size=3pt,inner sep=1pt,fill=magenta!80!white]
\tikzstyle{edge}=[draw,thick]
\foreach \i in {1,3,...,9} {
	\node[vertex,fill=green!80!white] (v\i) at ($ (0, 0) + ({36*\i + 0}:0.5) $) {};
}
\foreach \i in {11,13,...,19} {
	\node[vertex] (v\i) at ($ (3, 0) + ({36*\i + 0}:0.5) $) {};
	\pgfmathtruncatemacro{\j}{\i - 10}
	\path[edge] (v\i) -- (v\j);
}
\foreach \i in {10,12,...,18} {
	\node[vertex] (v\i) at ($ (1, -3) + ({36*\i + 180}:0.5) $) {};
}
\foreach \i in {0,2,...,8} {
	\node[vertex,fill=green!80!white] (v\i) at ($ (4,-3) + ({36*\i + 180}:0.5) $) {};
}
\path[edge] (v1) -- (v16);
\path[edge] (v9) -- (v14);
\path[edge] (v7) -- (v12);
\path[edge] (v5) -- (v10);
\path[edge] (v3) -- (v18);
\path[edge] (v6) -- (v16);
\path[edge] (v4) -- (v14);
\path[edge] (v2) -- (v12);
\path[edge] (v0) -- (v10);
\path[edge] (v8) -- (v18);
\path[edge] (v6) -- (v11);
\path[edge] (v4) -- (v19);
\path[edge] (v2) -- (v17);
\path[edge] (v0) -- (v15);
\path[edge] (v8) -- (v13);
\path[edge] (v1) -- (v3) -- (v5) -- (v7) -- (v9) -- (v1);
\path[edge] (v0) -- (v2) -- (v4) -- (v6) -- (v8) -- (v0);
\path[edge] (v11) -- (v15) -- (v19) -- (v13) -- (v17) -- (v11);
\path[edge] (v10) -- (v14) -- (v18) -- (v12) -- (v16) -- (v10);
\end{tikzpicture}
}
\caption{The smallest examples of QBN graphs. Kernel eigenvector entries are color-coded: one color represents entries $+1$,
while the other represents $-1$.}
\label{fig:smallQBN}
\end{figure}
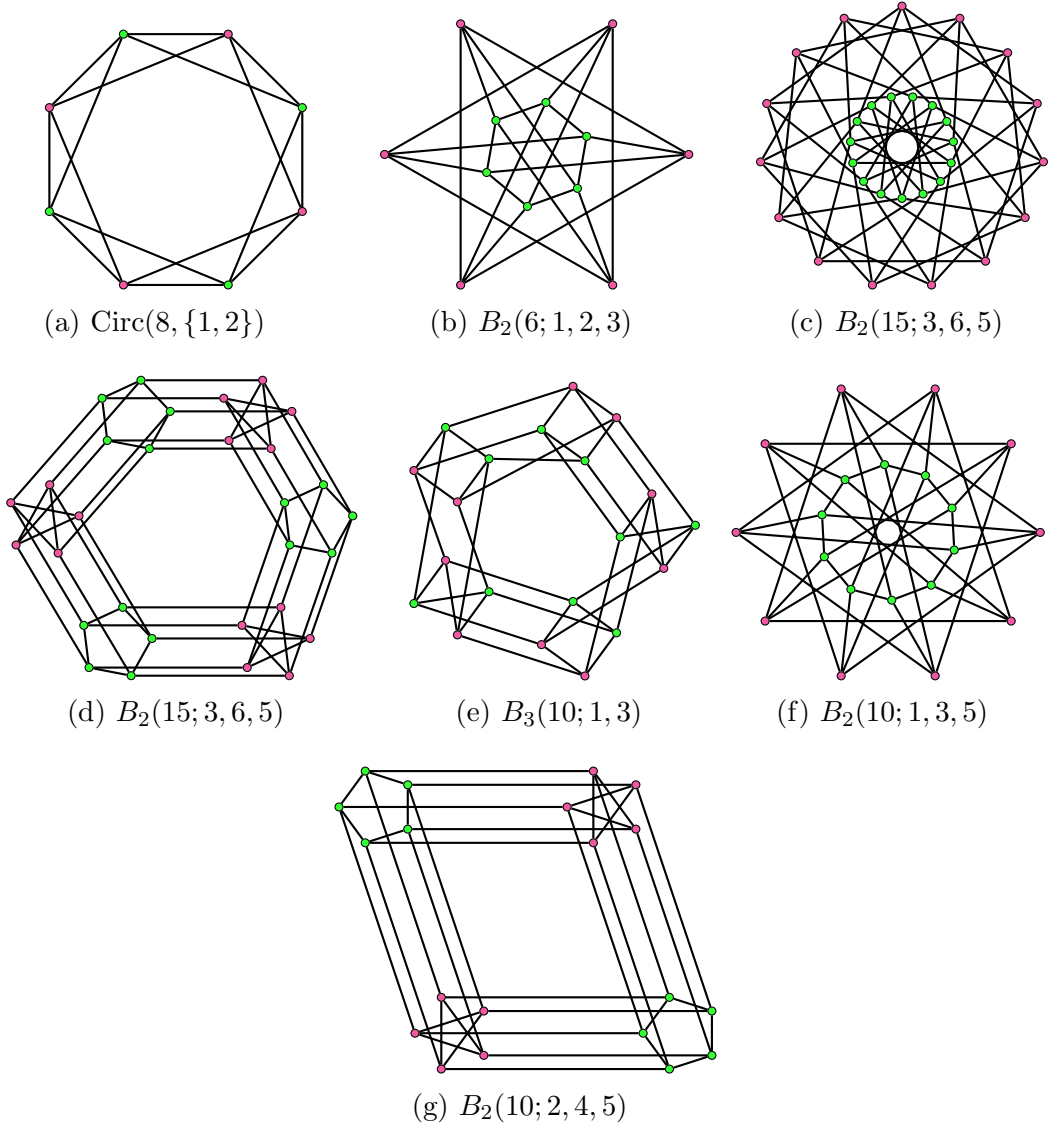

Observe that the QBN graphs from classes $\cB_1$ and $\cB_3$ are relatively rare in comparison with those from the class $\cB_2$. Furthermore, when $m = \frac{n}{2}$ is prime, then among the connected quartic bicirculants of class $\cB_2$, there is only one graph up to isomorphism that is not a nut graph, namely $B_2(m; 1, 1, 2)$. Indeed, in this case the conditions 2(ii), 2(iii) and 2(iv) from Theorem~\ref{main_theorem} are satisfied, so the graph $B_2(m; a, b, c)$ is a nut graph if and only if condition~2(i) also holds. Since $m$ is prime, we may assume without loss of generality that $a = 1$, and it is not difficult to further verify that $B_2(m; 1, b, c)$ is not a nut graph if and only if $b = 1$ and $c = 2$.

In Section \ref{sec:intro}, we noted the existence of quartic bicirculant graphs that simultaneously belong to more than one class from $\cB_1, \cB_2, \cB_3, \cB_4$. It is natural to pose the question whether a QBN graph can also belong to more than one class. The following proposition gives an affirmative answer.

\begin{proposition}\label{class_prop}
    For any $m \ge 6$ such that $m \equiv_4 2$, the following holds:
    \begin{enumerate}[1.]
        \item the graphs $B_2(m; 2, 2, \frac{m}{2})$ and $B_1(m; 2, 2)$ are isomorphic to $\Circ\left(2m, \left\{4, \frac{m}{2} \right\}\right)$ and it is a nut graph;
        \item the graphs $B_2(m; 1, 1, \frac{m}{2})$ and $B_3(m; 1, m - 1)$ are isomorphic to $\Circ\left(2m, \left\{ 2, \frac{m}{2} \right\}\right)$ and it is a nut graph.
    \end{enumerate}
\end{proposition}
\begin{proof}
    Note that $B_1(m; 2, 2)$ is a nut graph by Theorem \ref{main_theorem}. Now observe that $B_2(m; 2, 2, \frac{m}{2})$ and $B_1(m; 2, 2)$ are both isomorphic to $K_2 \sqr K_2 \sqr C_{m / 2} \cong C_4 \sqr C_{m / 2}$, which is, in turn, isomorphic to $\Circ\left(2m, \left\{4, \frac{m}{2} \right\}\right)$ because $4$ and $\frac{m}{2}$ are coprime.

    By Theorem \ref{main_theorem}, we also have that $B_3(m; 1, m - 1)$ is a nut graph. Observe that the mapping $f_1 \colon B_2(m; 1, 1, \frac{m}{2}) \to \Circ\left(2m, \left\{2, \frac{m}{2} \right\} \right)$ defined by
    \[
        f_1(x_i) = 2i \qquad \mbox{and} \qquad f_1(y_i) = 2i + \frac{m}{2}
    \]
    is an isomorphism. Also, the mapping $f_2 \colon B_3(m; 1, m - 1) \to \Circ\left(2m, \left\{2, \frac{m}{2} \right\} \right)$ defined by
    \[
        f_2(x_i) = \begin{cases}
            2i, & \text{if } 2 \mid i,\\
            2i + \frac{m}{2}, & \text{if } 2 \nmid i,
        \end{cases} \qquad \mbox{and} \qquad f_2(y_i) = \begin{cases}
            2i + \frac{m}{2}, & \text{if } 2 \mid i,\\
            2i, & \text{if } 2 \nmid i
        \end{cases}
    \]
    is an isomorphism. Therefore, $B_2(m; 1, 1, \frac{m}{2}) \cong B_3(m; 1, n-1) \cong \Circ\left(2m, \left\{2, \frac{m}{2} \right\} \right)$.
\end{proof}

The computational results from Table \ref{tab:enumeration} lead to the following natural question.

\begin{problem}
    Apart from the two families from Proposition \ref{class_prop}, are there any additional QBN graphs that belong to more than one class from $\cB_1$, $\cB_2$ and $\cB_3$?
\end{problem}

It would be interesting to study the possible automorphism groups of QBN graphs. For example, the five QBN graphs $B_2(10; 1, 2, 1)$, $B_2(10; 1, 1, 1)$, $B_2(10; 1, 2, 5)$, $B_2(10; 1, 3, 5)$ and $B_1(10; 2, 2)$ are of the same order and have mutually nonisomorphic (full) automorphism groups. The general question about the automorphism groups of quartic bicirculant graphs remains largely open, while some results have been obtained in special cases; see \cite{RWCayley, RWiso, RWvt, Dobson2022, quarticet}.

\section*{Acknowledgements}

I.\ Damnjanović is supported in part by the Science Fund of the Republic of Serbia, grant \#6767, Lazy walk counts and spectral radius of threshold graphs --- LZWK. N.\ Bašić is supported in part by the Slovenian Research Agency (research program P1-0294 and research project J5-4596). T.\ Pisanski is supported in part by the Slovenian Research Agency (research program P1-0294 and research projects J1-4351 and J5-4596). A.\ \v{Z}itnik is supported in part by the Slovenian Research Agency (research program P1-0294 and research projects J1-3002 and J1-4351).

\section*{Conflict of interest}

The authors declare that they have no conflict of interest.

\appendix

\pagebreak
\section{\texttt{SageMath} script for Theorem \ref{b2_main_th}}\label{comp_b2}

\begin{lstlisting}[language = Python, frame = trBL, escapeinside={(*@}{@*)}, aboveskip=10pt, belowskip=10pt, numbers=left, rulecolor=\color{black}]
#!/usr/bin/sage
import sys
from sage.all import *

# R.<x> = PolynomialRing(QQ)
R = PolynomialRing(QQ, names=('x',))
(x,) = R._first_ngens(1)

values = sorted([
    7, 14, 28, 56, 5, 10, 20, 40, 3, 6, 12, 24, 15, 30, 60, 120, 2, 4, 8, 
])

def v2(n):
    count = 0
    while Mod(n, 2) == 0:
        count += 1
        n = n // 2
    return count

def polynomial_q(f, a, b, c):
    return x**Mod(a + b, f) + x**Mod(a - b, f) + x**Mod(b - a, f) + \
        x**Mod(-a - b, f) - 2 - x**Mod(c, f) - x**Mod(-c, f)

for f in values:
    cyc = R.cyclotomic_polynomial(f)
    print(f'Processing {f}')

    for a in range(1, f + 1):
        for b in range(1, f + 1):
            for c in range(1, f + 1):
                # Condition (i)
                if gcd(gcd(f, a - b), c + a + b) != 1:
                    continue
                if gcd(gcd(f, a - b), c - a - b) != 1:
                    continue
                if gcd(gcd(f, a + b), c + a - b) != 1:
                    continue
                if gcd(gcd(f, a + b), c - a + b) != 1:
                    continue

                # Condition (ii)
                if v2(c) == v2(a) + 1 and v2(f) > v2(c):
                    continue
                if v2(c) == v2(b) + 1 and v2(f) > v2(c):
                    continue
 
                # Check divisibility
                q = polynomial_q(f, a, b, c)
                (quot, rem) = q.quo_rem(cyc)
                if rem.is_zero():
                    print(f, a, b, c)
\end{lstlisting}

\end{document}